\documentclass[11pt]{amsart}
\usepackage{amsmath, amsthm, amssymb}
\usepackage{amsmath,amscd}
\usepackage{mathabx}
\usepackage{hyperref}
\usepackage{mathrsfs}
\usepackage{xcolor, changepage} 
\usepackage{graphicx}
\usepackage{titletoc}
\usepackage{enumerate}
\usepackage{accents}

\usepackage{tikz}
\usetikzlibrary{matrix,arrows}
\usetikzlibrary{shapes}
\usetikzlibrary{calc}
\usetikzlibrary{arrows}
\usetikzlibrary{decorations.pathreplacing,decorations.markings}
\usepackage[all]{xy}
\usepackage{tikz-cd}

\usepackage{caption}
\usepackage{subcaption}
\usepackage{geometry}
\theoremstyle{plain}
\newtheorem{theorem}{Theorem}
\newtheorem{proposition}[theorem]{Proposition} 
\newtheorem{lemma}[theorem]{Lemma}
\newtheorem{remark}[theorem]{Remark}
\newtheorem{corollary}[theorem]{Corollary}
\newtheorem{definition}[theorem]{Definition}

\numberwithin{theorem}{section}

\DeclareMathOperator{\dist}{dist}

\DeclareMathOperator{\Int}{Int}
\DeclareMathOperator{\divv}{div}

\numberwithin{equation}{section}

\title[Foliation and Schoen's conjecture]{Foliation of area minimizing hypersurfaces in asymptotically flat manifolds and Schoen's conjecture}
\author{Shihang He}
\address{Key Laboratory of Pure and Applied Mathematics,
School of Mathematical Sciences, Peking University, Beijing, 100871, P. R. China
}
\email{hsh0119@pku.edu.cn}
\author{Yuguang Shi}
\address{Key Laboratory of Pure and Applied Mathematics,
School of Mathematical Sciences, Peking University, Beijing, 100871, P. R. China
}
\email{ygshi@math.pku.edu.cn}

\author{Haobin Yu}
\address{
School of Mathematics, Hangzhou Normal University, Hangzhou, 311121, P. R. China
}
\email{yhbmath@hznu.edu.cn}

\thanks{S. He, Y. Shi  are funded by the National Key R\&D Program of China Grant 2020YFA0712800. H. Yu is funded by NSFC12001147}


\begin{document}
\begin{abstract}
In this paper, we demonstrate that any asymptotically flat manifold $(M^n, g)$ with $4\leq n\leq 7$ can be foliated by a family of area-minimizing  hypersurfaces, each of which is asymptotic to  Cartesian coordinate hyperplanes defined at an end of  $(M^n, g)$. As an application of this foliation, we show that  for any asymptotically flat manifold $(M^n, g)$ with $4\leq n\leq 7$,  nonnegative scalar curvature and positive mass, the solution of free boundary problem for area-minimizing hypersurface in coordinate cylinder $C_{R_i}$ in $(M^n, g)$ either does not exist or drifts to infinity of $(M^n, g)$ as $R_i$ tends to infinity. Additionally, we  introduce a concept of globally minimizing hypersurface in $(M^n, g)$, and verify a version of the Schoen Conjecture.
\end{abstract}	
	\maketitle
	\tableofcontents
	\section{Introduction}

	A famous conjecture due to R.Schoen states that for a Riemannian manifold $(M,g)$ of dimension $3\le n\le 7$ and asymptotically flat (AF for short) of rate $\tau=n-2$ with non-negative scalar curvature, if it contains a non-compact area minimizing boundary, then $M$ is actually isometric to the flat $\mathbb{R}^n$.  We say a hypersurface $\Sigma$ in $(M^n,g)$ is area-minimizing if for any   hypersurface $\Sigma'$ in $(M^n,g)$ and a compact set $D\subset M$ we have 
	 $$
	 \mathcal{H}^{n-1}(\Sigma\cap D)\leq  \mathcal{H}^{n-1}(\Sigma'\cap D),	 
	 $$
	 provided the symmetric difference $\Sigma \triangle \Sigma'$ is contained in $D$. Here $\mathcal{H}^{n-1}(\cdot)$ denotes the $(n-1)$-dimensional Hausdorff measure. 
  The conjecture has been verified in dimension $n=3$ in \cite{CCE16}, see also \cite{Carlotto16} for the case where $4 \leq n < 8$ with a volume growth condition on the minimal hypersurface and \cite{AR1989} for splitting results of area-minimizing surface in $3$-dimensional manifolds with nonnegative Ricci curvature.  Recently, a version of Schoen conjecture was verified in  \cite{EK23}, and also in Theorem 5 in this paper, a family of properly and completely embedded area-minimizing boundaries, which serve as counterexamples to the conjecture, was constructed  in Schwarzschild manifolds for dimensions $4 \le n \le 7$. In other words, the usual area-minimizing condition is not strong enough to ensure the validation of Schoen conjecture. Inspired by this, we get  the following more general result.

	\begin{theorem}\label{thm: Existence}
		Let $(M^n,g)$ be an AF manifold with $4\leq n\leq 7$ and asymptotic order $\tau>\frac{n-1}{2}$, then there is $t_0>0$, such that for all $t\geq t_0$, there exists an area minimizing hypersurface $\Sigma_t$, satisfying:
		
		(1) Under the Cartesian coordinates of $M\setminus K$, $\Sigma_t$ is the graph of a function $u_t$ over  coordinate $z-$hyperplane $S_t$;  
		
		(2) There is a constant $C$ depends only on $k$, $n$, $\tau$ and $(M,g)$, so that for any $k\geq 1$, $\epsilon>0$, $|z-t|_{C^k}(x)\leq C|x|^{1-\tau-k+\epsilon}$ as $|x|\to +\infty$;
		
		(3) Moreover, except for countable  values $t_i$, $1=1,2,\cdots $, each of those area minimizing hypersurfaces satisfying conditions above is unique. Furthermore, $\{(x,z)\in M:|z|\geq t_0\}$  for some fixed $t_0$ is smoothly foliated by these $\Sigma_t$. 
	\end{theorem}
	
	In conjunction with Theorem 1.6 in \cite{CCE16},  Theorem \ref{thm: Existence} illustrates tremendous difference from the phenomenon in dimension 3. For an AF manifold $(M^n,g)$,  and let compact set $K\subset M$  be as defined in Definition \ref{def: AF manifold},  let $C_R$  be  the region bounded by the cylinder $\partial B^{n-1}_R(O)\times \mathbb{R}$ for $R>1$, where $B^{n-1}_R(O)$  is coordinate ball in a coordinate hyperplane at an end of  $(M^n,g)$. Then an application of Theorem \ref{thm: Existence} is given below, it may be of  some interest to compare this result  with Corollary 3 in \cite{Carlotto16}. Define 
\begin{equation}
\begin{split}
\mathcal{F}_R&:=\{\Sigma=\partial \Omega \cap \mathring{C_R}: \text{there is a Caccioppoli set } \Omega \subset M \text{ with locally finite perimeter}\\
	&\text{ and $C_R \cap \{z\leq b\}\subset \Omega $ and $\Omega \cap \{z\geq a\}=\phi  $ for some $-\infty< a\leq b<\infty$}, \}
	\end{split}\nonumber
\end{equation}

then we have 
	\begin{theorem}\label{thm: free boundary}
	Let $(M^n,g)$ be an AF manifold  as described  in Theorem \ref{thm: Existence} with nonnegative scalar curvature and $\tau>max \{\frac{n-1}{2}, n-3\}$. Suppose its ADM mass $m>0$. Then one of the following happens
		\begin{itemize}
		\item There exists $R_0>0$, such that for all $R>R_0$, there exists no hypersurface $\Sigma_i$ in $C_R$, which minimizes the volume in the $\mathcal{F}_R$;
		\item For any sequence $\{R_i\}$ tending to infinity such that there exists a hypersurface $\Sigma_{R_i}$ in $C_{R_i}$, which minimizes the volume in $\mathcal{F}_{R_i}$. Then $\Sigma_{R_i}$ drifts  to infinity,i.e. for any compact set $\Omega \subset M$, 	$\Sigma_{R_i}\cap \Omega=\phi$ for sufficiently large $i$.
		\end{itemize}
	\end{theorem}
		
 \begin{figure}
    \centering
    \includegraphics[width = 8cm]{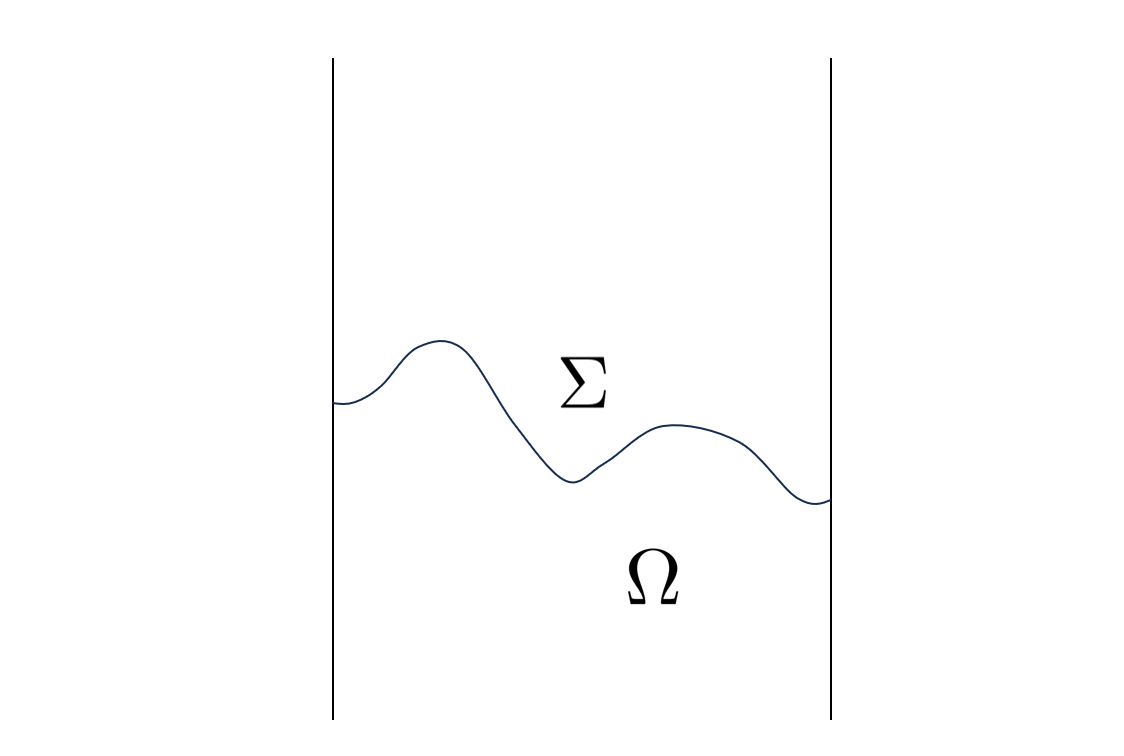}
    \caption{Free boundary hypersurface in the class $\mathcal{F}_R$}
    \label{f4}
\end{figure}    
  
	  Due to Theorem \ref{thm: free boundary}, it seems to us that  if there exists a hypersurface which minimizes the volume among all the hypersurfaces in some relative homology class represented by a given hypersurface $S_0$ then $M$ should be  actually flat. With these issues in mind, we introduce the following notion of global minimizing. 
	
	\begin{definition}\label{global min}
	Let  $(M^n,g)$ be an AF manifold. We say a hypersurface $\Sigma$ in $M$ is {\it globally minimizing}, if there exists a sequence of number $R_i\to +\infty$, such that 
		
		(1) $\Sigma$ intersects $\partial C_{R_i}$ transversally.
		
		(2) As free boundary hypersurface in $C_{R_i}$, $\Sigma_{R_i} = \Sigma\cap C_{R_i}$ minimizes the volume in  $\mathcal{F}_{R_i}$. Here and in the sequel, $S_t$ denotes the coordinated hyperplane $\{z:=x_{n}=t\}$ in $M$ outside $K$.
	\end{definition}

It is worth noting that $S_t$  is the only global minimizing hypersurface in $\mathbb{R}^n$. With this conception, we are able to demonstrate the following version of Schoen conjecture is true.

	\begin{theorem}\label{thm: cylinder minimizing}
		Let $(M,g)$ be an asymptotically flat Riemannian manifold of dimension $4\le n\le 7$ with non-negative scalar curvature and asymptotic order $\tau>max \{\frac{n-1}{2}, n-3\}$. If $M$ contains a globally minimizing hypersurface, then $M$ is isometric to the flat $\mathbb{R}^n$.
	\end{theorem}
 
To establish existence of $\Sigma_t$ in Theorem \ref{thm: Existence}, we first solve a Plateau problem in cylinder $C_r$ with boundary $\partial C_r \cap S_t$, let $\Sigma_{r,t}$ denotes the solution, and we aim for  $\Sigma_{r,t}$ to converge to $\Sigma_t$ as $r$ tends to the infinity. To achieve this, we need to demonstrate $\Sigma_{r,t}$ does not escape to infinity of $M^n$. We note that in the proof of positive mass theorem in \cite{SY79}, the assumption of negative ADM mass of the AF manifold ensures that the mean curvature of $S_t$ with respect to outward normal vector is positive for large $|t|$. Consequently,  $S_t$ can be served as barrier hyperplane for large $|t|$. In our current case, we lack such condition, however, we observe that the volume growth of $\Sigma_{r,t}$ can be controlled and therefore, its second fundamental forms and higher derivatives can be controlled as well, for details,  see Lemma \ref{lem: interior curvature estimate} below. 	Once the  second fundamental forms of interior part of $\Sigma_{r,t}$ are controlled, we are able to improve the Sobolev inequality on it, which can be used to deduced the estimate of $u:=z(x)-t$ on $\Sigma_{r,t}$. Here, $z(x)$ is the restriction of  coordinate function $x_n$  on  $\Sigma_{r,t}$,  see Proposition \ref{pro: C^0 estimate of u} below. Therefore, we see that  $\Sigma_{r,t}$ cannot  slide   off to infinity of $M^n$. Thus, by choosing a subsequence if necessary, we prove that $\Sigma_{r,t}$ converges to  $\Sigma_t$  as $r$ approaches to infinity. By careful analysis of minimal surface equation that $\Sigma_t$ satisfies, we obtain  its asymptotic behavior 	at infinity. Additionally, we show that 
	$\{(x,z)\in M:|z|\geq t_0\}$  for some large $t_0$ is smoothly foliated by these $\Sigma_t$. 	It may be of interest to compare this property with that of constant mean curvature surfaces near infinity of an AF manifold (c.f. Theorem 4.1 in \cite{HY1996}). 
	
	 We observe that if there is a  sequence $\{R_i\}$ tending to infinity and a solution $\Sigma_{R_i}$ for the free boundary problem in each  $C_{R_i}$ in the sense of (2) in Definition \ref{global min} with $\Sigma_{R_i}\cap \Omega\neq \phi$ for a fixed compact set $\Omega\subset M$, then by choosing a subsequence if necessary, we may assume $\{\Sigma_{R_i}\}$ converges locally smoothly to a complete area-minimizing hypersurface $\Sigma\subset M$. By the arguments in \cite{EK23}, it must be asymptotic to certain coordinate hyperplane $S_t$ in a end $E$ of $M$. Then by Theorem \ref{thm: Existence}, in this Cartesian coordinates determined  by $S_t$ at the end $E$, there is an area-minimizing hypersurface foliation whose leaf is asymptotic to a coordinate hyperplane $S_t$ for large $|t|$.  Then again by the arguments in \cite{EK23}, for any $p\in M$ with $|z(p)|$ large enough, we can demonstrate that there is a complete area-minimizing  hypersurface $\Sigma'$ which is asymptotic to some coordinate hyperplane $S_t$ and passes through $p$. Additionally,  $\Sigma'$ is  stable under asymptotically constant variation. On account of Proposition \ref{prop: uniqueness}, we see that $\Sigma'$ is must be $\Sigma_t$. Thus,  all of these minimal hypersurfaces $\Sigma_t$ are indeed stable under asymptotically constant variation, and hence, in conjunction with a suitable conformal deformation and  the positive mass theorem, each of them are flat,  totally geodesic and the Ricci curvature of $(M^n, g)$ with normal direction of $\Sigma_t$ vanishes along these minimal hypersurfaces. Therefore, by Proposition \ref{flatness}, $(M^n, g)$ is flat in the part with $|z|$ large enough. Together with the condition $\tau>n-3$, we obtain the ADM mass $m=0$ and reach the contradiction. Theorem \ref{thm: cylinder minimizing} is a direct conclusion of Theorem \ref{thm: free boundary}, thereby confirming  a version of Schoen conjecture. 
	
	The remainder of the paper is outlined as follows. In Section 2, we will introduce some notations and collect some facts on the properties of minimal surfaces. In Section 3, we will construct foliation of area minimizing hypersurfces. In Section 4, we will verify a version of Schoen conjecture. In  Appendix we will list some basic facts about $L^p$ estimate for  linear elliptic equations.

	\section{Preliminary}
 In this section, we will introduce some notations and collect some facts on the properties of minimal surfaces.

 	\begin{definition}\label{def: AF manifold}
		A $n$-dimensional Riemannian manifold $(M,g)$ is said to be asymptotically flat of order $\tau$ for some $\tau>\frac{n-2}{2}$, if there is a compact set $K$ in $M$, such that each component of $M\backslash K$ is diffeomorphic to $\mathbb{R}^n\backslash B^n_{1}(O)$, and the metric in $M\backslash K$ satisfies
		\begin{align*}
			|g_{ij}-\delta_{ij}|+|x||\partial g_{ij}|+|x|^2|\partial^2 g_{ij}| = O(|x|^{-\tau}),
		\end{align*}
		and $R\in L^1(M)$. Each component of $M\backslash K$ is called end of $(M^n,g)$.
  \end{definition} 
	
	 For $\alpha>0$, we use $O(|x|^{-\alpha})$ to denote the quantities which can be bounded by $C|x|^{-\alpha}$ for some $C$ depending only on $(M,g)$ throughout the paper. For a fixed number $t$ with $|t|>1$, denote $S_t$ to be the coordinate hyperplane of height $t$, {\it i.e.} $S_t = \mathbb{R}^{n-1}\times\{t\}$. In the end $E$ of an AF manifold $(M^n,g)$, any point $x$ has its coordinate $x=(x_1,x_2,\dots,x_n)$. We always call $x_n$ the $z$-component and denote $z(x) = x_n$
and use $D_{r,t}$ to denote the coordinate ball of radius $r$ in $S_t$, centered at the intersection of $S_t$ and the $z$-axis.
  
  \begin{lemma}\label{lem: interior curvature estimate}
		Let $(M^n,g)$ be an asymptotically flat manifold of order $\tau$ $(n\ge 4)$.  Then there is a hypersurface $\Sigma_{r,t}$  contained in $C_r$ minimizing the volume among all  hypersurfaces $S$ with  $\partial S=\partial D_{r,t}$. Moreover, for any $p\in \Sigma_{r,t}$ we have
		\begin{itemize}
			\item There exists a constant $\Lambda$ depending only on $(M^n, g)$ such that for any $s>1$ with $\partial \Sigma_{r,t}\cap B^n_s(p)=\phi$ and
     \begin{equation}\label{eq: area growth}
     \mathcal{H}^{n-1}(\Sigma_{r,t}\cap B^n_s(p))\leq\Lambda s^{n-1}.
     \end{equation}	
      Here and in the sequel, $ B^n_s(p)$ denotes the geodesic ball in $(M^n, g)$ with centre $p$ and radius $s$.     \item For any $k\in \mathbb{Z}^+$, $\alpha\in (0, 1)$, there is a constant $\beta>0$ depends only on $k, \alpha$, $\Lambda$ and $n$ so that
     \begin{equation}\label{eq: curvature estiamte}
     	|h|_{C^{k,\alpha}(\Sigma_{r,t}\cap B^n_{\frac{\rho}{2}}(p))} \leq \frac{\beta}{{\rho}^{k+\alpha}} ,
     \end{equation}
      provided  $\partial \Sigma_{r,t}\cap B^n_{\rho}(p)=\phi$. Here and in the sequel $ \mathcal{H}^{n-1}(\cdot)$ and $h$ denote the $(n-1)-$dimensional Hausdorff measure and the second fundamental form of $ \Sigma_{r,t}$ respectively. 
        \end{itemize}
	\end{lemma}
 \begin{remark}
 Indeed, by the proof of Lemma \ref{lem: interior curvature estimate},  we will know that $\Sigma_{r,t}$ is a part of a boundary of compact domain in $M^n$.	
 \end{remark}

	\begin{proof}[Proof of Lemma \ref{lem: interior curvature estimate}]
		Let $O_t$ be the intersection of $S_{r,t}$ and the $z$-axis, then $B^n_r(O_t)$ is a ball enclosing $D_{r,t}$. By asymptotic flatness of $M$ we   know that  $\partial B^n_r(O_t)$ is  mean convex with respect to outer normal  vector for $r$ sufficiently large, this enable us to find the solution $\Sigma_{r,t}$ of the Plateau problem  of mininal boundary   with  $\partial \Sigma_{r,t}=\partial D_{r,t}$ (c.f. Theorem 1.20 in \cite{Giu1984}). Due to its minimality, we know that for all $s$  	with $\partial \Sigma_{r,t}\cap B^n_s(p)=\phi$ we have

$$
 \mathcal{H}^{n-1}(\Sigma_{r,t}\cap B^n_s)\leq \frac{1}{2}\mathcal{H}^{n-1}(\partial B^n_s) \leq \Lambda s^{n-1}.
 $$

To verify \eqref{eq: curvature estiamte} it suffices to show

 $$
\sup_{x\in \Sigma_{r,t}\cap B^n_{\rho}(p)} (\rho-d(x,p))|h|(x)\leq \beta<\infty,
 $$

here $d(x, p)$ denotes the distance function of $x$ and $p$ in $(M^n, g)$. Indeed, this is direct conclusion of Corollary 1 in \cite{Schoen1981}. However, for the convenience of readers,  we will utilize the standard point-picking arguments to verify this.  Suppose the above inequality false, then there is a sequence $\{\Sigma_{r_i,t_i}\}$, $\{ B^n_{\rho_i}(p_i) \}$ and $\{\beta_i\}$ which tends to  infinity with 

$$
\sup_{x\in \Sigma_{r_i,t_i}\cap B^n_{\rho_i}(p_i)} (\rho_i-d(x,p_i))|h|(x)=\beta_i,
 $$
and $\partial\Sigma_{r_i,t_i}\cap B^n_{\rho_i}(p_i)=\phi$, let us assume 

$$
(\rho_i-d(q_i,p_i))|h|(q_i)=\sup_{x\in \Sigma_{r_i,t_i}\cap B^n_{\rho_i}(p_i)} (\rho_i-d(x,p_i))|h|(x)=\beta_i,
$$
 and $2l_i:=\rho_i-d(q_i,p_i)$. There are three cases we need to concern:

 {\bf Case 1}:  $\{l_i\}$ is bounded while $\{|h|(q_i)\}$ tends to infinity; 
 
 {\bf Case 2}: $\{l_i\}$ tends to infinity while $\{|h|(q_i)\}$ is bounded;
 
 {\bf Case 3}: Both $\{l_i\}$  and $\{|h|(q_i)\}$ tends to infinity.

It is enough to show Case 1,  since Case 2 and Case 3 can be handled by the similar arguments.  Let us consider $\Sigma_i:=\Sigma_{r_i,t_i}\cap B^n_{l_i}(q_i) $, then for any $y\in \Sigma_i$ there holds
$$
|h|(y)\leq 2 |h|(q_i):=2 \lambda_i.
$$
Note that $(M^n, g)$ is AF, $\{\lambda_i\}$ tends to infinity, we know that $\{(M^n, \lambda^2_i g, q_i)\}$ locally smoothly converges to $\mathbf{R}^n$ in Gromov-Hausdorff sense, meanwhile, $\{\Sigma_i,\lambda^2_i g|_{\Sigma_i}, q_i \}$ locally smoothly converges to a  complete and properly embedding  stable minimal hypersurface $\Sigma $ in  $\mathbf{R}^n$. Without loss of the generality, we amy $\{q_i\}$ converges to $O\in \Sigma$, then the norm of the second fundamental forms of $\Sigma$ at the point $O$  is $1$. On the other hand, by the definition, we know that  $\Sigma$ enjoys the Euclidean volume growth, then due to Theorem 1 in \cite{Be2023}, $S$ must be totally geodesic, therefore,  we get the contradiction.  
\end{proof}
In the rest of this subsection, we aim to establish suitable Sobolev type inequality on compact minimal hypersurface $\Sigma$ in asymptotically flat manifold, which will be used in the next section. 
     
 \begin{lemma}\label{lem: estimate}
 Assume $\Sigma$ is a compact minimal hypersurface in $(M^n, g)$ satisfying \eqref{eq: area growth}, and there is a coordinate ball $B_{R}$ of radius $R$, such that $\Sigma\subset B_R$, $\partial\Sigma\subset \partial B_R$. Given $\alpha>n-1$ and $r_1>1$, then for any $r>0$ and  $t$ with $|t|>1$, there holds
 \begin{align}\label{decay estimate}
     \int_{\Sigma\backslash B^n_{r_1}}|x|^{-\alpha}d\bar{\mu} \le Cr_1^{(n-1-\alpha)},
 \end{align}
  for some $C$ depending only on $(M^n, g)$.
 \begin{proof}
     Choose $k\in \mathbb{N}_+$ such that $\frac{R}{2^{k+1}}<r_1\le \frac{R}{2^k}$. Then
	\begin{align*}
			\int_{\Sigma_{r,t}\backslash B^n_{r_1}}|x|^{-\alpha}d\bar{\mu}
			=&\sum_{i=0}^{k}\int_{(\Sigma\cap (B^n_{\frac{R}{2^i}}\backslash B^n_{\frac{R}{2^{i+1}}})\backslash B^n_{r_1}}|x|^{-\alpha}d\bar{\mu}\\
			\le&\sum_{i=0}^{k}(\frac{R}{2^{i+1}})^{-\alpha}\mathcal{H}^{n-1}(\Sigma\cap B^n_{\frac{R}{2^i}})\\
			\le&\sum_{i=0}^{k}\Lambda(\frac{R}{2^{i+1}})^{-\alpha}(\frac{R}{2^i})^{n-1}\\
			=& \Lambda\sum_{i=0}^{k}(2^{k-1-i}r_1)^{-\alpha}\cdot (2^{k+1-i}r_1)^{n-1}\\
			\leq& 2^n\Lambda\sum_{i=0}^{k}2^{(n-1-\alpha)(k-1-i)}r_1^{n-1-\alpha} \\
			<& 2^{n+1}\Lambda r_1^{n-1-\alpha}.
		\end{align*}
		The conclusion follows.
 \end{proof}
 \end{lemma}

 \begin{figure}
    \centering
    \includegraphics[width = 15cm]{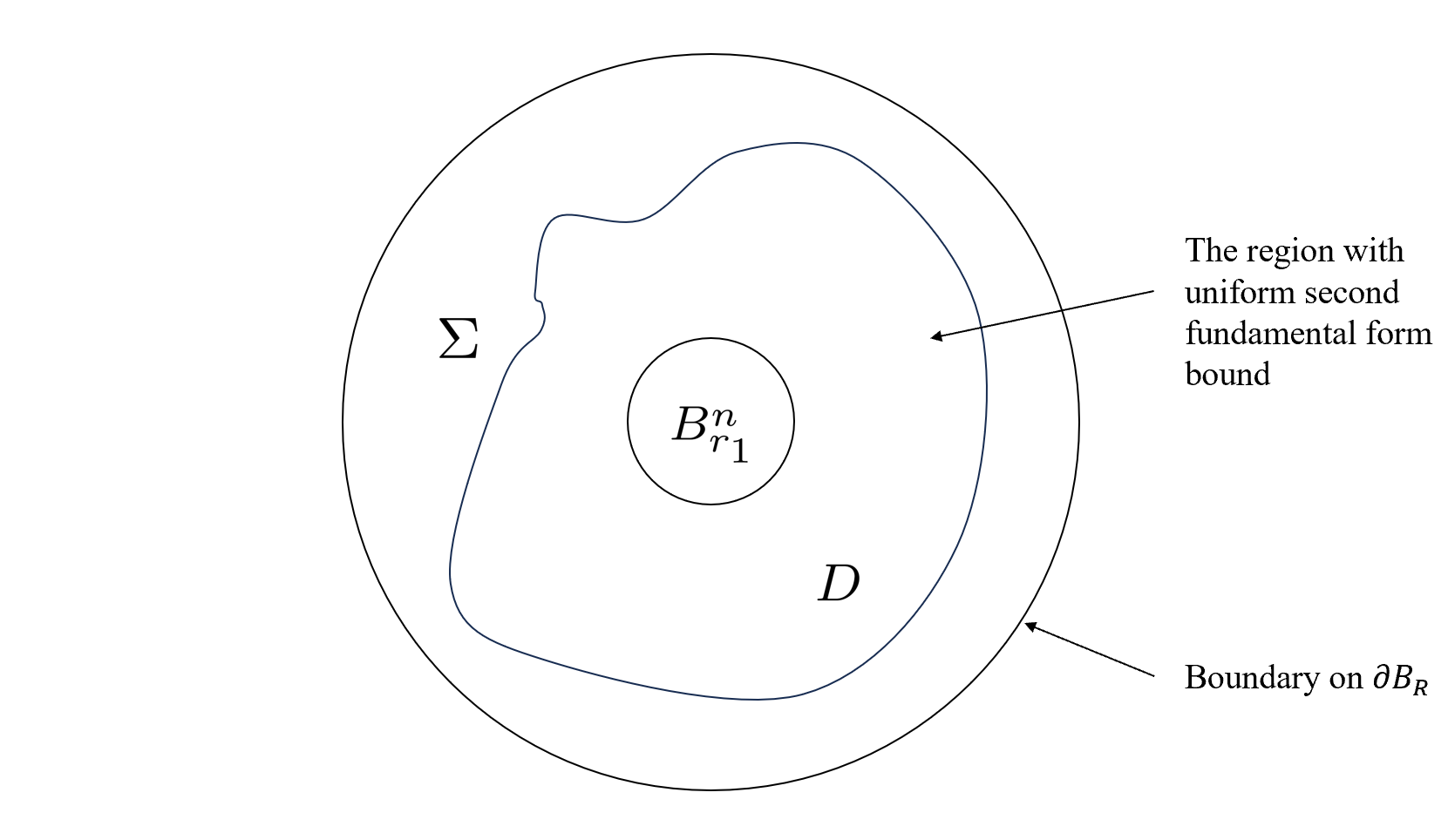}
    \caption{Region $D$ in $\Sigma$}
    \label{f7}
\end{figure}   
	Throughout the paper, we use $\bar{g}$ to denote the Euclidean metric on $\mathbb{R}^n$ and a bar to indicate that a geometric quantity is computed with respect to $\bar{g}$. If we regard $\Sigma$ as a hypersurface in $(\mathbb{R}^n, \bar{g})$, then we have
	\begin{lemma}\label{lem: L^m estimate of H}
	Under the same assumption as in Lemma \ref{lem: estimate}, given $r_1>0, \delta\in(0,1)$, assume further there is a region $D\subset\Sigma$ satisfying $|h|_g\leq \frac{C}{\delta}$ on $D$ for some $C$ depending only on $(M^n, g)$. There holds
	\begin{align}
		\int_{D\backslash B^n_{r_1}}|\bar{H}|^md\bar{\mu} \le C\delta^{-m}r_1^{m(1-\tau)}
	\end{align}
	Here $m=n-1$ and $C$ depends only on $(M,g)$.
	\end{lemma}
	
	\begin{proof}
	As $\Sigma$ is a minimal surface with $|h|_g\leq C\delta^{-1}$ on $D$,	
 by
  \[\bar{H}=H+O(|x|^{-1-\tau})+O(|x|^{-\tau}|h|),
  \]
  it follows that $|\bar{H}|=O(\delta^{-1}|x|^{-\tau})$ on $D$. Since $\tau>\frac{n-2}{2}\geq 1$, then by Lemma \ref{lem: estimate},
		the conclusion follows.
	\end{proof}
 
	Next, we recall the following Sobolev-type inequality due to Michael and Simon [MS73] on hypersurfaces in Euclidean space.
	\begin{lemma}\label{lem: Sobolev inequality 1}
		Let $\bar{H}$ be the mean curvature of a hypersurface $\Sigma$  in $(\mathbb{R}^{m+1}, \bar{g})$.  Then there exists $c_m$ depending only on $m$ such that  for any positive Lipschitz function $u$ with compact support on $\Sigma$, there holds
		\begin{align}\label{eq: Sobolev1}
			(\large\int_{\Sigma}u^{\frac{m}{m-1}}d\bar{\mu}\large)^{\frac{m-1}{m}}\le c_m\int_{\Sigma}(|\bar{\nabla} u| + |\bar{H}|u)d\bar{\mu}
		\end{align}
	\end{lemma}
	Now we are ready to establish the following Sobolev inequality on $\Sigma$.
	\begin{lemma}\label{lem: Sobolev inequality 2}
	Let $\Sigma$ be an embedding minimal hypersurface in $(M^n,g)$ as in Lemma \ref{lem: estimate}, given $\delta\in(0,1)$ with $|h|_g\leq \frac{C}{\delta}$ on a region $D\subset \Sigma$ for some $C$ depending only on $(M^n, g)$,
     	then there exists some  $r_1\geq \delta^{\frac{-1}{\tau-1}}$ such that for any smooth function $v$ with compact support on
 $D\backslash B^n_{r_1}$
 there holds
		\begin{align}\label{eq: Sobolev2}
			(\int_{D\backslash B^n_{r_1}} |v|^{\frac{2m}{m-2}}d\mu)^{\frac{m-2}{2m}}\le C(\int_{D\backslash B^n_{r_1}}|\nabla v|^2d\mu)^{\frac{1}{2}},
		\end{align}
   where $C$ depends only on $(M^n ,g)$.
	\end{lemma}
	\begin{proof}
		We assume $v>0$ first. Let $u = v^{\frac{2(m-1)}{m-2}}$ be a function on $D$ supported outside $B^n_{r_1}$. Plugging this into \eqref{eq: Sobolev1}, we get
		\begin{align*}
			(\int_{D\backslash B^n_{r_1}} v^{\frac{2m}{m-2}}d\bar{\mu})^{\frac{m-1}{m}}
			\le & c_m\frac{2(m-1)}{m-2}\int_{D\backslash B^n_{r_1}} |\bar{\nabla} v|\cdot v^{\frac{m}{m-2}}d\bar{\mu}+\int_{D\backslash B^n_{r_1}} v^{\frac{2(m-1)}{m-2}}|\bar{H}|d\bar{\mu}\\
			\le & c_m\frac{2(m-1)}{m-2}\int_{D\backslash B^n_{r_1}} |\bar{\nabla} v|\cdot v^{\frac{m}{m-2}}d\bar{\mu} + (\int_{D\backslash B^n_{r_1}}|\bar{H}|^md\bar{\mu})^m(\int_{D\backslash B^n_{r_1}}v^{\frac{2m}{m-2}}d\bar{\mu})^{\frac{m-1}{m}}
		\end{align*}
		As $\tau>\frac{n-2}{2}\geq 1$, then by Lemma \ref{lem: L^m estimate of H} we can choose some $r_1\geq \delta^{\frac{-1}{\tau-1}}$ such that 
		\begin{align*}
			\int_{D\backslash B^n_{r_1}}|\bar{H}|^md\bar{\mu}\le \frac{1}{2}.
		\end{align*}
		It follows that
		\begin{align*}
			(\int_{D\backslash B^n_{r_1}} v^{\frac{2m}{m-2}}d\bar{\mu})^{\frac{m-1}{m}}
			\le& c_m\frac{2(m-1)}{m-2}\int_{D\backslash B^n_{r_1}} |\bar{\nabla} v|\cdot h^{\frac{m}{m-2}}d\bar{\mu}\\
			\le &C_m(\int_{D\backslash B^n_{r_1}}|\bar{\nabla} v|^2d\bar{\mu})^{\frac{1}{2}}(\int_{D\backslash B^n_{r_1}}v^{\frac{2m}{m-2}}d\bar{\mu})^{\frac{1}{2}}.
		\end{align*}
		Here $C_m$ depends only on $m$. As $(M^n, g)$ is asymptotically flat, then there exist some $C$ depending on $(M^n, g)$ such that
  \[ 
  \frac{d\bar{\mu}}{C}\leq d\mu\leq C d\bar{\mu} \ \ \text{and}\ \ \ 
  \frac{|\bar{\nabla} v|}{C}\leq |\nabla v|\leq C |\bar{\nabla} v|
  \]
  Hence,
  \begin{align*}
			(\int_{D\backslash B^n_{r_1}} v^{\frac{2m}{m-2}}d\mu)^{\frac{m-2}{2m}}\le C(\int_{D\backslash B^n_{r_1}}|\nabla v|^2d\mu)^{\frac{1}{2}},
		\end{align*}
  This shows the Sobolev inequality holds for positive $C^1$-smooth function  on $D\backslash B^n_{r_1}$. The same thing also holds true for positive function in $W^{1,\frac{2m}{m-2}}$ by approximating argument. For general smooth function $v$, this can be proved by verifying the inequality for positive and negative part of $v$ respectively, and the conclusion follows. 
	\end{proof}

 \section{Construction of foliations of the area minimizing  hypersurfaces in AF manifolds}
In this section we construct area minimizing foliation, discuss existence result in Theorem \ref{thm: Existence},  uniqueness result in  Proposition \ref{prop: uniqueness}. Finally,  we are able to show a smooth foliation structure for our AF manifold  in Sec. 3.3.

\subsection{Constructing area minimizing hypersurfaces}
  Denote $z(x)$ to be the $n$-th coordinate function in $\mathbb{R}^n\backslash B^n_{1}(O)$ and set $u=z(x)-t$ 
  for some $t$. Let $\Sigma$ be any embedded area minimizing minimal surface.  We derive the equation of $u$ over $\Sigma$ as follows. 
    \begin{lemma}\label{lem: intrinsic MSE}
    Let $u$ and $\Sigma$ be described as above, then
            \begin{align}\label{eq: mse}
            \Delta_{\Sigma}u = f,
        \end{align}
        for smooth function $f$ on  $\Sigma\backslash B^n_{1}(O)$ with $|f| = O(|x|^{-\tau-1})$.
    \end{lemma}

    \begin{proof}
        Denote $\nu$ to be the normal vector along $\Sigma$. By minimality of $\Sigma$, we have 
        \begin{align*}
            \Delta_{\Sigma}u =& \Delta_{M^n}u-
            H\frac{\partial u}{\partial\nu}- \nabla^2_{M^n}u(\nu,\nu)\\
            =&g^{ij}(\partial_{i}\partial_{j}z-\Gamma_{ij}^{k}\partial_{k}z)- \nu^{i}\nu^{j}(\partial_{i}\partial_{j}z-\Gamma_{ij}^{k}\partial_{k}z)\\
            =& O(|x|^{-\tau-1}).
        \end{align*}
    \end{proof}
    Let $\Sigma, D$ be as in Lemma \ref{lem: Sobolev inequality 2} with $\partial D\subset S_t$ for some $t>0$. Let $u$ be a smooth function defined on $\Sigma$ satisfying
\begin{equation}\label{eq: poisson eq}
\left\{
\begin{aligned}
\Delta u&=f  \ \ \text{in} \  D\\
u&=0\ \text{on}\  \partial D,
\end{aligned}
\right.
\end{equation}
for some  $f=O(|x|^{-\tau-1})$. We start by establishing  $L^p(p>1)$ estimate for $u$.
\begin{lemma}\label{lem: L^p estimate for u}
            Suppose $\tau>\frac{m}{2}:=\frac{n-1}{2}$. 
            Then for some  $r_1\geq \delta^{\frac{-1}{\tau-1}}$ we have
		\begin{align*}
		    (\int_{D\backslash B^n_{r_1}} |u|^{\frac{2m}{m-2}}d\mu)^{\frac{m-2}{2m}} \le Cr_1^{\frac{m}{2}}(t+r_1),
		\end{align*}
  for some $C$ depending only on  $(M^n,g)$. Furthermore, if
 $\Sigma\cap B^n_{r_1}=\emptyset$, then
		\begin{align*}
		    (\int_{D} |u|^{\frac{2m}{m-2}}d\mu)^{\frac{m-2}{2m}} \le Cr_1^{-\frac{2m(\tau+1)}{m+2}+m}.
		\end{align*}
  
	\end{lemma}

 \begin{proof}

 Let $0\leq\zeta\leq 1$ be a cut off function on $\Sigma$ satisfying
     \begin{align*}
         &\zeta = 1 \mbox{ outside } B^n_{r_1}\\
         &\zeta = 0 \mbox{ inside } B^n_{r_1/2}\\
         &|\nabla\zeta|<\frac{4}{r_1}
     \end{align*}
     By integration by parts we have
\begin{equation}\label{eq: 10}       
   \begin{split}
       \int_{D\backslash B^n_{r_1}} |\nabla(\zeta u)|^2d\mu 
             &\le -\int_{D\backslash B^n_{r_1}} \zeta^2u\Delta u d\mu + \int_{D\backslash B^n_{r_1}} |\nabla\zeta|^2u^2 d\mu\\
             &=-\int_{D\backslash B^n_{r_1}} \zeta^2ufd\mu + \int_{D\backslash B^n_{r_1/2}} |\nabla\zeta|^2u^2 d\mu\\
         &\le  (\int_{D\backslash B^n_{r_1}} |\zeta u|^{\frac{2m}{m-2}}d\mu)^{\frac{m-2}{2m}}(\int_{D\backslash B^n_{r_1}} |\zeta f|^{\frac{2m}{m+2}}d\mu)^{\frac{m+2}{2m}}\\
         &+\mathcal{H}^m(D\cap B^n_{r_1/2})(t+r_1)^2
         \end{split}   
    \end{equation} 
    Here we have utilized 
    $$ 
  |u|\leq (t+r_1) \text{ on $D \cap B^n_{r_1}$}.
   $$    
    Since 
    $$f=O(|x|^{-\tau-1}),$$
    then by Lemma \ref{lem: estimate}, we have
     \begin{align}\label{eq: 12}
         \int_{\Sigma\backslash B^n_{r_1}} |\zeta f|^{\frac{2m}{m+2}}d\mu
         = O(r_1^{-\frac{2m(\tau+1)}{m+2}+m}) 
     \end{align}
     which tends to zero as $r_1$ increases, provided $-\frac{2m(\tau+1)}{m+2}+m<0$, {\it i.e.} $\tau >\frac{m}{2}$. Combining this with Lemma \ref{lem: Sobolev inequality 2} and \eqref{eq: 10} we obtain the desired result.\\
    If $\Sigma\cap B^n_{r_1}=\emptyset$, then by (\ref{eq: poisson eq}), Lemma \ref{lem: Sobolev inequality 2} and integration by parts we have
     \begin{align*}
         (\int_{D} u^{\frac{2m}{m-2}}d\mu)^{\frac{m-2}{m}}\le& \int_{D} |\nabla u|^2d\mu = \int_{D} -uf d\mu\\
         \le &  (\int_{D} |u|^{\frac{2m}{m-2}}d\mu)^{\frac{m-2}{2m}}(\int_{D} |f|^{\frac{2m}{m+2}}d\mu)^{\frac{m+2}{2m}}.
     \end{align*}
     The conclusion follows from \eqref{eq: 12}. 
 \end{proof}
    Now we use  Lemma \ref{lem: L^p estimate for u} and Moser iteration  to establish the $C^0$ estimate for $u$.

 \begin{proposition}\label{pro: C^0 estimate of u}
     There is some fixed $r_1\geq 1$ such that for any $r,t>0$ we have
     $$
     \|u\|_{C^{0}(\Sigma_{r,t}\backslash B^n_{r_1})}\leq C r_1^{\frac{m}{2}}(t+r_1),
     $$
      for some $C$ depending only on $(M^n,g)$. Moreover, if $\Sigma_{r,t}$ does not intersect $B^n_{r_1}$, then the $C^0$ norm of $u$ is bounded by some uniform constant $C(t)$. In particular, if $\tau>\frac{m}{2}$, then $C(t)$ approaches to zero as $|t|$ tends to infinity.
\end{proposition}

 \begin{proof}
      We  give an upper bound for $u$ on $\Sigma_{r,t}$ and the lower bound can be derived in the same way. Note $u=0$ on $\partial\Sigma_{r,t}$. Choose some $a\in[1,2]$ such that $\Sigma^+_{r,t}=\Sigma_{r,t}\cap\{z(x)\geq a+t\}$ is a regular hypersurface. Without loss of generality, we may assume $\Sigma^+_{r,t}=\Sigma_{r,t}\cap\{z(x)\geq 1+t\}$.  Then for any $x\in\Sigma^+_{r,t}$ and $y\in\partial\Sigma_{r,t}$, we have
     \[
d_{\Sigma_{r,t}}(x,y)\geq d_{M^n}(x,y)\geq\frac{1}{2}|z(x)-z(y)|\geq\frac{1}{2}.
     \]
The last second inequality is due to the asymptotic flatness of metric. By Lemma \ref{lem: interior curvature estimate}, for any $r,t>0$, $\Sigma^+_{r,t}$ satisfies \eqref{eq: area growth} and $|h|\leq C$ for some uniform $C$ depending only on $(M^n, g)$. From this we see $\Sigma^+_{r,t}$ satisfies the conditions satisfied by $D$ in previous lemmas.
     It suffices to consider  $\Sigma_{r,t}\cap\{z(x)\geq 3+t\}\neq\emptyset$, otherwise we immediately have $u\leq 3$.
    For any point $q\in \Sigma_{r,t}\cap\{z(x)\geq 3+t\}$, we have $\dist_{\Sigma_{r,t}}(q,\partial \Sigma^+_{r,t})\geq 1$. 
    Set $v=u-a$ and Consider the equation
    \begin{equation}\label{eq: poisson eq2}
\left\{
\begin{aligned}
\Delta v&=f  \ \ \text{in} \  \Sigma^+_{r,t}\\
v&=0\ \ \text{on} \ \  \partial\Sigma^+_{r,t}.
\end{aligned}
\right.
\end{equation}
\begin{figure}
    \centering
    \includegraphics[width = 15cm]{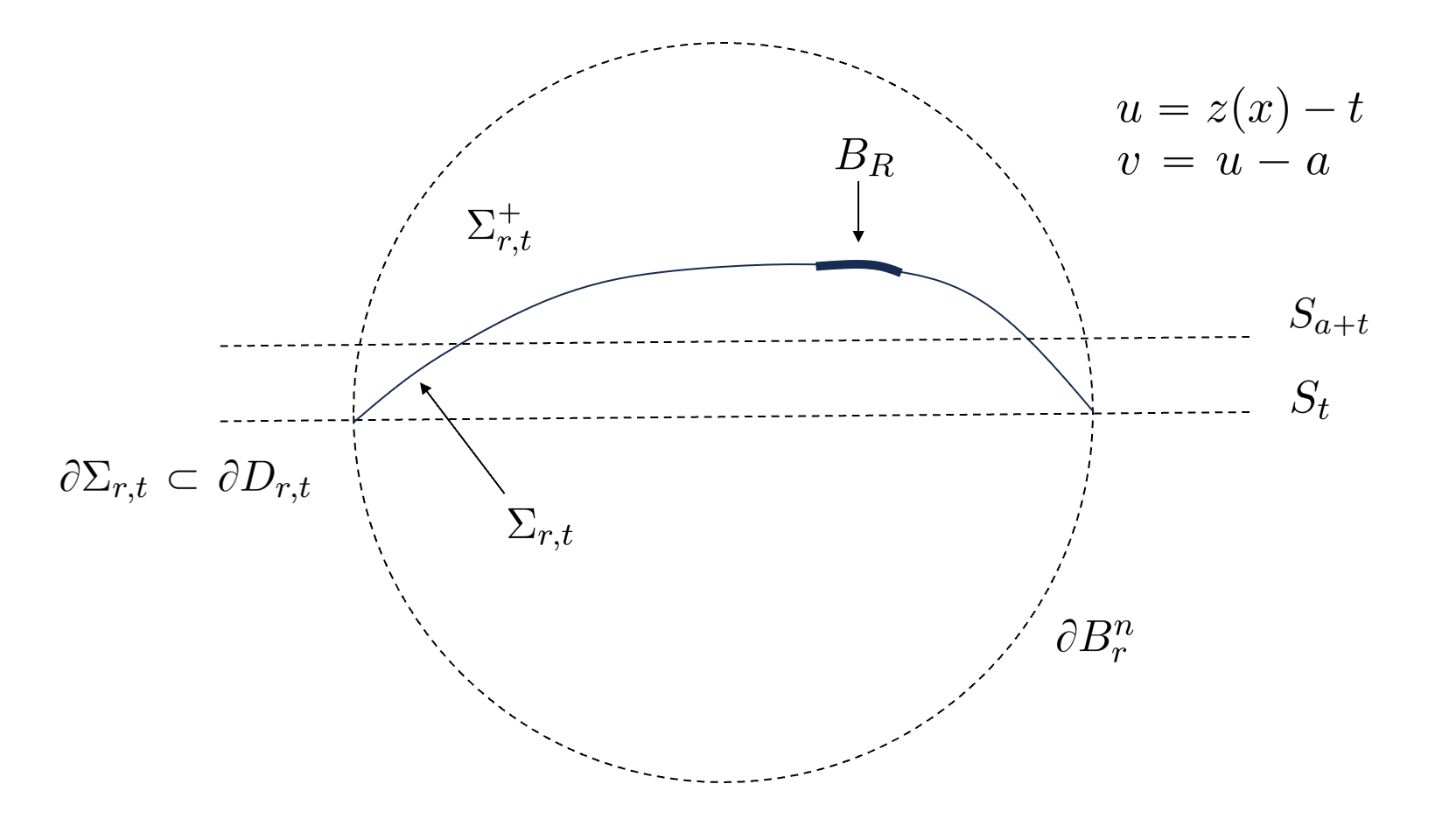}
    \caption{$C^0$ estimate of $u$: $\Sigma = \Sigma_{r,t}$, $D = \Sigma_{r,t}^+$}
    \label{f8}
\end{figure}   

    Denote $B_R$ to be the intrinsic geodesic ball of radius $R$ centered at $q$ throughout the proof. Let $\zeta$ be a compact supported function on $B_R$. For $p>2$, multiply $\zeta^2v^{2p-1}$ 
    on \eqref{eq: poisson eq2} and integration by part, we obtain
     \begin{equation}\label{eq: 13}
         \begin{split}
             -\int_{B_R}\zeta^2v^{2p-1}f =& \int_{B_R}\nabla(\zeta^2v^{2p-1})\cdot\nabla v\\
         =&(2p-1)\int_{B_R}\zeta^2v^{2p-2}|\nabla v|^2 + 2\int_{B_R}\zeta v^{2p-1}\nabla\zeta\cdot\nabla v
         \end{split}
     \end{equation}
     By Schwartz inequality
     \begin{align*}
         &2\int_{B_R}|\zeta v^{2p-1}\nabla\zeta\cdot\nabla v|\\
         \le& 2(\int_{B_R}\zeta^2 |v|^{2p-2}|\nabla v|^2)^{\frac{1}{2}}(\int_{B_R}v^{2p}|\nabla\zeta|^2)^{\frac{1}{2}}\le \frac{2p-1}{2}\int_{B_R}\zeta^2 |v|^{2p-2}|\nabla v|^2+\frac{2}{2p-1}\int_{B_R}v^{2p}|\nabla\zeta|^2
     \end{align*}
     Combining with \eqref{eq: 13} gives
     \begin{align*}
         \int_{B_R}\zeta^2v^{2p-2}|\nabla u|^2\le \frac{4}{(2p-1)^2}(\int_{B_R}v^{2p}|\nabla\zeta|^2-\frac{2p-1}{2}\int_{B_R}\zeta^2v^{2p-1}f)
     \end{align*}
     Applying Sobolev inequality for $\zeta v^p$ on $B_R$, we have
     \begin{equation}\label{eq: sobolev}
         \begin{split}
             (\int_{B_R}|\zeta v^p|^{\frac{2m}{m-2}})^{\frac{m-2}{m}}
             \le& C\int_{B_R}|\nabla\zeta v^p|^2\\
             \le&2C\int_{B_R}|\nabla\zeta|^2u^{2p}+2Cp^2\int_{B_R}\zeta^2v^{2p-2}|\nabla v|^2\\
             \le& 2C\int_{B_R}|\nabla\zeta|^2v^{2p}+
            10C\int_{B_R}v^{2p}|\nabla\zeta|^2+8Cp\int_{B_R}\zeta^2v^{2p-1}|f|\\
            \le &12C\int_{B_R}v^{2p}|\nabla\zeta|^2+8Cp\int_{B_R}\zeta^2v^{2p-1}|f|.
         \end{split}
     \end{equation}
     By Holder inequality,
     \begin{equation}\label{eq: Holder}
         \begin{split}
             \int_{B_R}\zeta^2 v^{2p-1}|f|\le& 
             (\int_{B_R}\zeta^2v^{2p})^\frac{2p-1}{2p}
             (\int_{B_R}\zeta^2|f|^{2p})^{\frac{1}{2p}}\\
             \le&\int_{B_R}\zeta^2v^{2p}+\frac{1}{2p}\int_{B_R}\zeta^2|f|^{2p}.
         \end{split}
     \end{equation}
     Substituting \eqref{eq: Holder} into \eqref{eq: sobolev} gives
     \begin{equation}\label{eq: iteration eq}
         \begin{split}
             (\int_{B_R}|\zeta v^p|^{\frac{2m}{m-2}})^{\frac{m-2}{m}}
         \le& 12C\int_{B_R}v^{2p}|\nabla\zeta|^2+8Cp\int_{B_R}\zeta^2v^{2p}+4C\int_{B_R}\zeta^2|f|^{2p}
         \end{split}
     \end{equation}
    Let $\zeta_k$ be the cut-off function supported on $B_{R_k}$ with
     \begin{align*}
         \zeta_k = 1 \mbox{ on }B_{R_{k+1}} \ \text{and}\ \  |\nabla\zeta_k|<\frac{2}{R_k-R_{k+1}} = 2^{k+2},
     \end{align*}
     where $R_k = (\frac{1}{2}+\frac{1}{2^k})$.
      Set 
     \begin{align*}
        p_k = 2(\frac{m}{m-2})^{k} \ \ \text{and}\ \  
        I_k = \int_{B_{R_k}}|v|^{p_k}+\int_{B_{R_k}}|f|^{2p_k}
     \end{align*}
     As $B_{R_0}$ lies outside $B^n_{r_1}$ and it's not hard to see $|f|_{C^0(B_{R_0})}\leq 1$.
     Then it follows from \eqref{eq: iteration eq} that
     \begin{equation}\label{eq: iteration2}
     \begin{split}
        I_{k+1}\le& C_1 [4^k+(\frac{m}{m-2})^k]I_k^{\frac{m}{m-2}}+C_1\int_{B_{R_k}}|f|^{2p_{k+1}}\\
         \le & C_2 4^kI_k^{\frac{m}{m-2}}.
     \end{split}
     \end{equation}
It's easy to show that
\[
I_{k+1}^{\frac{1}{p_{k+1}}}\leq C_3I_1^{\frac{m-2}{2m}}.
\]
Sending $k$ to $\infty$, in conjunction with Lemma \ref{lem: L^p estimate for u} gives
\begin{equation}\label{eq: C^0 bound of u}
    v(q)\leq C_4[(\int_{B_R} v^{\frac{2m}{m-2}})^{\frac{m-2}{2m}}
+(\int_{B_R} f^{\frac{2m}{m-2}})^{\frac{m-2}{2m}})]
\le C_5r_1^{\frac{m}{2}}(t+r_1).
\end{equation}
If $\Sigma_{r,t}$ does not intersect $B^n_{r_1}$, then by Lemma \ref{lem: L^p estimate for u} we have 
\[
(\int_{B_R} v^{\frac{2m}{m-2}})^{\frac{m-2}{2m}}\leq Cr_1^{-\frac{2m(\tau+1)}{m+2}+m}.
\]
As 
\[(\int_{B_R} f^{\frac{2m}{m-2}})^{\frac{m-2}{2m}}\leq Cr_1^{-\frac{2m(\tau+1)}{m+2}+m},
\]
then by \eqref{eq: C^0 bound of u}, we obtain the uniform bound for the $C^0$ norm of $u$ and this bound tends to zero as $|t|\rightarrow \infty$.
\end{proof}

With Proposition \ref{pro: C^0 estimate of u} we're able to construct area minimizing hypersurfaces in $(M,g)$. For fixed $t$, by Lemma \ref{pro: C^0 estimate of u},  $u_i(x) = z(x)-t_i$ is uniformly bounded on $\Sigma_{r_i,t}$, which shows $\Sigma_{r_i,t}$ intersects with fixed compact set. Hence, by taking a subsequence if necessary, $\Sigma_{r_i,t}$ converges to an area minimizing hypersurface $\Sigma_t$ locally smoothly  as $r_i\to +\infty$.  Moreover, $\Sigma_t$ is between  two parallel hyperplanes. Note that  a complete minimal hypersurface between two parallel hyperplane in $\mathbb{R}^n$ must be an affine hyperplane,  and an area minimizing hypersurface $\Sigma$ in $\mathbb{R}^n$ with $\partial\Sigma = S_0\cap B^n_1(O)$ must be the unit disk in $S_0$, in conjunction this with a blow-down argument we know that,  outside a compact set, $\Sigma_t$ must be a graph over $S_t$.

\subsection{Asymptotic behavior of $\Sigma_t$}

Let $\Sigma_t$ be as above. Let $u$ be the graph function of $\Sigma_t$ over $S_t$ outside a compact set. 
First, we prove the uniform gradient bound of the graph function $u$ outside give compact set if $\Sigma_t$ lies between two parallel hyperplanes. We use $(x', x^n)$ to denote a point $x\in M^n \setminus B^n_1(O)$. 

\textbf{Notational Remark}

Here and in the sequel, we may sometimes use abuse of notation. For $x\in\Sigma_t$, $u = z(x)-t$ denotes a function on $\Sigma_t$, as what was used in Lemma \ref{pro: C^0 estimate of u}. When $\Sigma_t$ can be locally written as a graph over $S_t$, we may use $u$ to denote the graph function. Note that in two definitions $u$ has the same value at corresponding points: $u(x)$ in the first definition equals $u(x')$ in the second definition, when $x = (x',x^n)$.
\begin{lemma}\label{lem: small gradient outside compact set}
Suppose that $\Sigma_t$ can be bounded by $S_{t\pm \lambda}$ for some $\lambda\geq 1$. Then there is a constant $C$ depending only on $(M^n,g)$ such that for $|x'|\gg1$ we have
 \[
    |\nabla u(x')|\leq \frac{C\lambda}{|x'|}.
    \]
\end{lemma}
\begin{proof}
  We first show $|\nabla u(x')|\to 0$ as $|x'|\to \infty$. Suppose not,  then we can identify  a constant $\delta_0>0$ and a sequence of points $p_i=(x_i',x_i^n)$ with $|x_i'|\to \infty$  satisfying
    \begin{align}\label{eq: initial decay for gradient}
        |\nabla u|(x_i')\geq\delta_0
    \end{align}
     By asymptotic flatness of $(M^n, g)$, we know that $(M, g, p_i)$ converges locally smoothly to $(\mathbb{R}^n, \bar g,  O)$ in the Gromov-Hausdorff's sense  as $i\to +\infty$. Here $\bar g$ denotes the standard Euclidean metric.  Additionally, we get that $(\Sigma_t, p_i)$ converge to an area minimizing graph $\Sigma$ with graph function $\bar{u}$ passing through the origin $O$ in $\mathbb{R}^n$.  Since $n\leq 7$, then $\Sigma$ must be a hyperplane. Combining with that  $\Sigma_{t}$ is bounded by $S_{t\pm \lambda}$ we deduce that $\bar{u}=0$.  This is not compatible with \eqref{eq: initial decay for gradient}.
    
   Next, we use scaling down argument to improve the estimate. 
    Take $p=(x',t)\in S_t$ with $|x'|=2\sigma\gg1$ and consider the region
    $\Omega_{\sigma}$ which is the part of the slab $\mathbb{R}^{n-1}\times[t-\lambda, t+\lambda]$ within the cylinder $C_{\sigma}(x')$ centered
around the line $\{(x',t) : |x'| = 2\sigma\}$. 
    Rescale the asymptotically flat metric $g_{ij}$  to the metric
  $\tilde {g}_{ij}$ using the
  transformation $\Phi$ given by $x = p + \sigma \tilde x$, i.e. 
  \[
  \tilde g =\Phi^*g.
  \]
The rescaled minimal surface is denoted by $\tilde{\Sigma}_t$ and
it can be bounded by $\tilde{S}_{\pm\lambda \sigma^{-1}}$.
We choose $\sigma$ large enough such that $\tilde{\Sigma}$ is a smooth minimal surface with $\tilde{u}$ satisfying the uniform elliptic equation. Then by the interior estimate, we have
\[
|\tilde{u}|_{C^{2,\alpha}(B_\frac{1}{2})}\leq C (|\tilde{u}|_{C^0(B_1)}+|\tilde{f}|_{C^{0,\alpha}(B_1)})
\leq C  \frac{\lambda}{\sigma},
\]
where the norm is taken with respect to $\tilde{g}$. Rescale back and use the fact $|\nabla u|$ is rescale invariant we get the desired decay estimate for $|\nabla u|$. 
\end{proof}
\begin{proposition}\label{pro:C^0 asymtotic behavior}
  Given $\Sigma_t$  as above, there is some $C=C_t$ depending only on $(M^n,g)$ and $t$ such that for any  $\epsilon>0$, 
    $|z(x)-t|\leq C_t|x'|^{1-\tau+\epsilon}$ as $|x'|\to +\infty$.
\end{proposition}
\begin{proof}
    As $\Sigma_t$ is bounded by $S_{\pm Ct}$ from above and below and
    we have the decay estimate for $\nabla u$, 
    then we can use Proposition 9 in \cite{EK23} to show that  $\Sigma_t$ must asymptotic to some hyperplane $S_{t'}$, furthermore, we have 
    \begin{equation}\label{eq: asymptotic of graph}
        |z(x)-t'|\leq C_{t'}|x'|^{1-\tau+\epsilon} \ \ \text{as}\ \  |x'|\to +\infty.
    \end{equation}
    Here $ C_{t'}$ is a positive constant depending only on $(M,g)$ and $t'$.
    
    \textbf{Claim}\ \ {$t\leq t'$}\\
    Suppose not, then we can choose a sequence of $\{r_i\}$ with $\Sigma_{r_i,t}$ converge to the given $\Sigma_t$ in the $C^{\infty}_{c}$ sense as $i\to \infty$.  Set
    \[\delta=\frac{t'-t}{4} \ \ \text{and}\ \ 
     \ \Sigma^{\delta}_{i,t}=\Sigma_{r_i,t}\cap \{z(x)\geq t+2\delta\}.
    \]
Then
\begin{align*}
		    (\int_{\Sigma^{\delta}_{i,t}\backslash B^n_{r_1}} |z(x)-(t+2\delta)|^{\frac{2m}{m-2}}d\mu)^{\frac{m-2}{2m}} \le Cr_1^{\frac{m}{2}}(t+r_1).
		\end{align*}
Sending $i\to\infty$  gives 
\begin{align*}
(\int_{(\Sigma_{t}\cap \{z(x)\geq t+2\delta\})\backslash B^n_{r_1}} |z(x)-(t+2\delta)|^{\frac{2m}{m-2}}d\mu)^{\frac{m-2}{2m}} \le Cr_1^{\frac{m}{2}}(t+r_1).
\end{align*}
It follows from \eqref{eq: asymptotic of graph} that $z(x)=t+4\delta+O(|x'|^{1-\tau+\epsilon})$.
Hence, we get the desired contradiction. By a similar argument, we may show $t\geq t'$. Therefore, we get the conclusion.  
\end{proof}

For future reference,  we introduce the following definition
\begin{definition}\label{def: asyhyperplane}
    A complete noncompact area minimizing  hypersurface $\Sigma$ is called asymptotic to some hyperplane $S_t$ if 
    $\Sigma$ is the graph over $S_t$ outside some compact set and  $|z(x)-t|\to 0$ as $x\to \infty$.
\end{definition}
\begin{remark}\label{re: asymptotic to hyperplane}
    If $\Sigma$ is a complete noncompact area minimizing  hypersurface as in Definition \ref{def: asyhyperplane}, then by the arguments of the proof of Lemma \ref{lem: interior curvature estimate}, its second fundamental form is uniformly bounded. Therefore, One can show the Sobolev inequality \eqref{eq: Sobolev2} holds on $\Sigma \setminus K$.Here $K$ is a compact set of $M$. If $\Sigma$ is  asymptotic to some hyperplane $S_t$
    then we can  use Proposition 9 in \cite{EK23} again to show that  for any  $\epsilon>0$
    it holds $|z(x)-t|\leq C|x'|^{1-\tau+\epsilon} \ \ \text{as}\ \  |x'|\to +\infty.$
\end{remark}
Lemma \ref{pro:C^0 asymtotic behavior} implies that for any $t\in\mathbb{R}$, the area minimizing hypersurface asymptotic to $S_t$. In the rest of this subsection, we aim to investigate more general property of area minimizing hypersurface asymptotic to $S_t$, so we always assume that $\Sigma$  is a complete noncompact area minimizing  hypersurface  asymptotic to some hyperplane $S_t$. We will give the explicit analysis for  the asymptotic behavior of $\Sigma$.

\begin{lemma}\label{lem: L^p estimate for graph}
 Let $\Sigma$ be any complete area minimizing hypersurface that is asymptotic to $S_t$,  suppose $\tau>\frac{m}{2}:=\frac{n-1}{2}$.  Then $u=z(x)-t$ is $L^{\frac{2m}{m-2}}$-integrable on  $\Sigma\backslash B^n_{r_1}$. Moreover,  there exists some constant $C$ depending only on $(M^n, g)$ such that
\begin{align*}
(\int_{\Sigma\backslash B^n_{r_1}} |u|^{\frac{2m}{m-2}}d\mu)^{\frac{m-2}{2m}} \le Cr_1^{\frac{m}{2}}(t+r_1).
\end{align*} 
\end{lemma}
 \begin{proof}
It follows from Remark \ref{re: asymptotic to hyperplane} that $u=z(x)-t$ is $L^{\frac{2m}{m-2}}$-integrable on  $\Sigma\backslash B^n_{r_1}$.
 Given $\theta>0$, Set $\Sigma^+_{t,\theta}=\Sigma\cap\{z(x)-t\geq \theta\}$ and $\Sigma^-_{t,\theta}=\Sigma\cap\{z(x)-t\leq -\theta\}$. Then $\Sigma^+_{t,\theta}$ and $\Sigma^+_{t,\theta}$ are compact.
By applying Lemma \ref{lem: L^p estimate for u} in the case that $\Sigma = \Sigma_t\cap B_R^n$ for $R$ sufficiently large and $D = \Sigma^+_{t,\theta}$, we obtain
		\begin{align*}
		    (\int_{\Sigma^+_{t,\theta}\backslash B^n_{r_1}} |u-\theta|^{\frac{2m}{m-2}}d\mu)^{\frac{m-2}{2m}} \le Cr_1^{\frac{m}{2}}(t+r_1),
		\end{align*} 
  Similarly,
  \begin{align*}
		    (\int_{\Sigma^-_{t,\theta}\backslash B^n_{r_1}} |u+\theta|^{\frac{2m}{m-2}}d\mu)^{\frac{m-2}{2m}} \le Cr_1^{\frac{m}{2}}(t+r_1).
		\end{align*} 
  Thus
  \[
(\int_{\Sigma^+_{t,\theta}\backslash B^n_{r_1}} |u-\theta|^{\frac{2m}{m-2}}d\mu)^{\frac{m-2}{2m}}+
(\int_{\Sigma^+_{t,\theta}\backslash B^n_{r_1}} |u+\theta|^{\frac{2m}{m-2}}d\mu)^{\frac{m-2}{2m}} \le Cr_1^{\frac{m}{2}}(t+r_1).
  \]
The assertion follows by  sending $\theta\to 0$ and using Fatou Lemma.
 \end{proof}
Let $\{\Sigma_i\}$ be a sequence of area minimizing  minimal surfaces asymptotic to $\{S_{t_i}\}$ with 
$t_i\to +\infty$. 

\begin{proposition}\label{pro:drift}
$\{\Sigma_i\}$  must drift off to the infinity as $i\to +\infty$.
\end{proposition}

\begin{proof}
    We take the contradiction argument. 
    Suppose not, then after passing to a subsequence(we still denote by $\{\Sigma_{i}\}$)
    we have $\Sigma_{i}\cap K\ne \emptyset$ for some compact $K$. Then $\Sigma_{i}$ converges locally smoothly to an area minimizing boundary $\Sigma$. Let $R>r_1$ be a constant to be fixed.  By Lemma \ref{lem: L^p estimate for graph}, for some uniform $C_0$ depending only on $(M^n, g)$, we have
\begin{align}\label{eq:uniform L^p}
(\int_{\Sigma_{i}\cap(B^n_R\backslash B^n_{r_1})} |u_i|^{\frac{2m}{m-2}}d\mu)^{\frac{m-2}{2m}}
\le C_0r_1^{\frac{m}{2}}(t_i+r_1)
\end{align}
where $u_i$ is defined by $u_i=z(x)-t_i$.
Since $\Sigma_{i}\cap B^n_R$ converges  smoothly to $\Sigma\cap B^n_R$, we have that for $i$ large enough 
\begin{equation}\label{eq:lower bound for L^p}
\begin{split}
(\int_{\Sigma_{i}\cap(B^n_R\backslash B^n_{r_1})} |u_i|^{\frac{2m}{m-2}}d\mu)^{\frac{m-2}{2m}}
&\geq (t_i-R)\left(\mathcal{H}^m(\Sigma_{i}\cap (B^n_R\backslash B^n_{r_1}))\right)^{\frac{m-2}{2m}}\\
&\geq\frac{1}{2}(t_i-R)\left(\mathcal{H}^m(\Sigma\cap (B^n_R\backslash B^n_{r_1}))\right)^{\frac{m-2}{2m}}
\end{split}
\end{equation}
where we have used $|u_i|=|z(x)-t_i|\geq t_i-R$ in $\Sigma_{i}\cap (B^n_R\backslash B^n_{r_1})$. Fix some $R\gg r_1^{2m}$ such that 
 $\mathcal{H}^m(\Sigma\cap (B^n_R\backslash B^n_{r_1}))\geq C_0r_1^{m}$.   
Sending $i$ to the infinity, we find \ref{eq:lower bound for L^p} is not compatible with \ref{eq:uniform L^p}. Hence, we get the desired contradiction.
\end{proof} 
Now we are ready to show

\begin{proposition}\label{pro: entire graph}
    There exists some $t_0$ such that for any area minimizing   hypersurface $\Sigma'_t$ asymptotic to $S_t$  for some $t\geq t_0$ it holds
    \begin{equation}\label{eq: estimate for normal vector}
     \inf_{x\in \Sigma'_t}\langle \nu(x),\partial z\rangle\geq\frac{1}{2}
    \end{equation}
 and  $\Sigma'_t$  must be the entire graph over $S_t$ .
\end{proposition}

\begin{proof}    
 We take the contradiction argument. Suppose not, then we can choose a sequence of
 $t_i\to\infty$, a sequence of area minimizing  hypersurfaces
 $\{\Sigma'_{t_i}\}$ asymptotic to $\{S_{t_i}\}$ and points $p_i\in\Sigma'_{t_i}$ such that
 $\langle\nu(p_i),\partial z\rangle\leq\frac{1}{2}$.
Using the similar argument as in Lemma \ref{lem: small gradient outside compact set},  $(M,g, p_i)$ converges to $(\mathbb{R}^n,\bar g, O)$  and  $(\Sigma'_{t_i}, p_i)$ converges locally smoothly to an area minimizing hypersurface passing through the origin in $\mathbb{R}^n$ as $i\to +\infty$. Since $\Sigma'_{t_i}$ is asymptotic to $S_{t_i}$, then $\Sigma$ must be a hyperplane vertical to $\partial z$.   This is not compatible with the fact that 
 $\langle\nu(O),\partial z\rangle\leq\frac{1}{2}$.

It suffices  to show the projection map $\pi:\Sigma'_t\to S_t$ given by 
    \[
    \pi: (x',x^n)\to (x',t)\in S_t\ \ \text{for}\ \ x=(x',x^n)\in M^n
    \]
is injective for $t\geq t_0$.
It follows from \eqref{eq: estimate for normal vector} that projection map $\pi:\Sigma'_t\to S_t$ is the local diffeomorphism. Choose $R_t\gg1 $ and let
$C_{R_t}$ be the cylinder consisting of the points $x=(x',x^n)$ satisfying $|x'|\leq R_t$.
Constraint on $\Sigma'_{t}\cap C_{R_t}$ and $S_{t}\cap C_{R_t}$, we have 
\[
\pi:\Sigma'_{t}\cap C_{R_t}\to S_{t}\cap C_{R_t} 
\]
is  local diffeomorphism with $\pi:\partial(\Sigma'_{t}\cap C_{R_t})\to \partial(S_{t}\cap C_{R_t})$
being diffeomorphism as $\Sigma'_t$ being asymptotic to $S_t$. Since $\Sigma'_{t}\cap C_{R_t}$ is compact, $\pi|_{\Sigma'_{t}\cap C_{R_t}}$ is the covering map to $S_{t}\cap C_{R_t}$. As $S_{t}\cap C_{R_t}$ is simply connected, it follows that 
projection map $\pi:\Sigma'_t\to S_t$ is injective.

\end{proof}

Let $\Sigma_t'$ be any area minimizing surface asymptotic to $S_t$ for some $t\geq t_0$. As $(M^n, g)$ is asmptotically flat, we can take some fixed $t_0\gg1$ such  that $\frac{1}{C_0}\delta_{ij}\leq g|_{\Sigma'_t}\leq C_0\delta_{ij}$ for some uniform $C_0$ and $\Sigma'_t\cap B^n_{r_1}=\emptyset$ where $r_1$ is determined by Proposition \ref{pro: C^0 estimate of u}. Next, we want to prove the following more precise decay estimate for higher order derivative of the graph function $u$ over $S_t$ when $t$ is sufficiently large.

\begin{lemma}\label{lem: L^p decay rsp t}
    Let $t_0$ be taken as above and $\Sigma_t'$ be any area minimizing surface asymptotic to $S_t$ for some $t\geq t_0$. 
   Then for $p>\frac{2m}{m-2}$ there exists a function $C(t)$ relying only on $p$ and $(M^n,g)$ with $C(t)\to 0$ as $t\to +\infty$, such that
    \begin{align}
			||u||_{W^{2,p}} < C(t)
    \end{align}
    Here, the $W^{2,p}$ norm could be computed either on $S_t$ or on $\Sigma_t$.
    More generally, let $\mathcal{U}$ be an open set containing $K$, then there is a constant $C$ depends only on $n,p,(M,g)$, $\mathcal{U}$ and $K$ with
  \begin{align}
			||u||_{W^{2,p}(\Sigma_t\setminus \mathcal{U})} < C,
    \end{align}
 
and $C$ tends to zero when  $\mathcal{U}$ exhausts $M$. 
\end{lemma}
\begin{proof}
    Note that  $\Sigma'_t$ could be written as an entire graph over $S_t$, $u$ could be regarded as a function on both $\Sigma'_t$ and $S_t$. In the following, we may not distinguish our notation for simplicity, and the domain of definition of $u$ could be identified by context. By calculation in Lemma \ref{lem: intrinsic MSE} we get
    \begin{align*}
        \Delta_{\Sigma'_t}u = f,\quad f = O(|x|^{-\tau-1})
    \end{align*}
 Note that the Sobolev inequality holds on $\Sigma'_t$ with some uniform Sobolev constant. By Proposition 9
 in \cite{EK23}, for any $\varepsilon>0$, it holds $|\nabla u|\leq C|x|^{-\tau+\varepsilon}$ as $x\to \infty$.
 It follows that $|\nabla u|^2\in L^1(\Sigma'_t)$.   Then integrating by parts gives
    \begin{align*}
		    (\int_{\Sigma'_{t}} |u|^{\frac{2m}{m-2}}d\mu)^{\frac{m-2}{m}} \le C_m\int_{\Sigma'_t} |\nabla u|^2 d\mu \le C_m(\int_{\Sigma'_t} |f|^{\frac{2m}{m+2}} d\mu)^{\frac{m+2}{2m}}(\int_{\Sigma'_t}u^{\frac{2m}{m-2}}d\mu)^{\frac{m-2}{2m}}
		\end{align*}
  Since $\Sigma'_t$ drift off to the infinity as $t\to\infty$, by Lemma \ref{lem: estimate} we have 
  \[
\int_{\Sigma'_t} |f|^{\frac{2m}{m+2}} d\mu\to 0\ \ \text{as}\ \ t\to\infty.
  \]
  Hence
  \begin{align*}
       (\int_{\Sigma'_{t}} |u|^{\frac{2m}{m-2}}d\mu)^{\frac{m-2}{2m}}\leq C(t)
       \ \ \text{with}\ \ C(t)\to 0\mbox{ as }t\to \infty
  \end{align*}
    Now we consider $u$ as graph function defined on $S_t$. Note that the volume element $d\mu$ is uniformly $C^0$ equivalent to that on $S_t$, we get
    \begin{align*}
               (\int_{S_{t}} |u|^{\frac{2m}{m-2}}dy)^{\frac{m-2}{2m}} < C(t)
                \ \ \text{with}\ \ C(t)\to 0\mbox{ as }t\to \infty
    \end{align*}
    By Proposition \ref{pro: C^0 estimate of u}, for $p>\frac{2m}{m-2}$ there holds
    \begin{align*}
        (\int_{S_{t}} |u|^{p}dy)^{p} < C(t), \ \ \text{with}\ \ C(t)\to 0\mbox{ as }t\to \infty\\
        (\int_{\Sigma'_{t}} |u|^{p}d\mu)^{p} < C(t), \ \ \text{with}\ \ C(t)\to 0\mbox{ as }t\to \infty
    \end{align*}

    Next, we derive the $W^{2,p}$ estimate of $u$ over $\Sigma_t$. We cover $S_t\cong \mathbb{R}^m$ by a collection of balls $B_{\alpha}$ of radius $1$, centered at each interger point on $S_t$. Then by covering $S_t$ by concentric balls $\Tilde{B}_{\alpha}$ of radius $2$, the multiplicity of each point is no greater than $4^m$. Consequently, $\Sigma_t$ could be covered by $\pi^{-1}(B_{\alpha})$ with multiplicity no greater than $4^m$. By uniform gradient estimate Lemma \ref{pro: entire graph}, there is a uniform constant $L_0$ bounding the diameter of each $\pi^{-1}(B_{\alpha})$ and $\pi^{-1}(\Tilde{B}_{\alpha})$. By $L^p$ interior estimate, we have for each $\alpha$ that
    \begin{align*}
        ||u||_{W^{2,p}(\pi^{-1}(B_\alpha))}\le C(||u||_{L^p(\pi^{-1}(\Tilde{B}_{\alpha}))}+||f||_{L^p(\pi^{-1}(\Tilde{B}_{\alpha}))})
    \end{align*}
    Therefore,
    \begin{align*}
        ||u||_{W^{2,p}} &\le \sum_{\alpha} ||u||_{W^{2,p}(\pi^{-1}(B_\alpha))}\\
        &\le \sum_{\alpha}C(||u||_{L^p(\pi^{-1}(\Tilde{B}_{\alpha}))}+||f||_{L^p(\pi^{-1}(\Tilde{B}_{\alpha}))})\\
        &\le 4^mC(||u||_{L^p}+||f||_{L^p}) < C(t),\quad C(t)\to 0\mbox{ as }t\to \infty
    \end{align*}
    Finally, by uniform gradient estimate, the same thing also holds true when considering $u$ as the function defining on $S_t$. This concludes the proof of Lemma \ref{lem: L^p decay rsp t}.
\end{proof}

Once $ W^{2,p}$-norm of $u$ is obtained, we get $C^{1,\alpha}$-estimate of $u$ by choosing $p$ large enough, and then by Schauder estimate we finally have $C^{k,\alpha}$-estimate of $u$ for all $k\geq 1$. With these facts, we establish the following lemma in terms of  constant $C$ independent on $t$, hence,  improve Proposition \ref{pro:C^0 asymtotic behavior}.

 \begin{lemma}\label{lem: decay estimate of u}
Let $t_0$ be as in Lemma \ref{lem: L^p decay rsp t} and $\Sigma'_t$ be any area minimizing hypersurface asymptotic to  $S_t$ for some $t>t_0$. Then there is a constant $C$ depends only on $n$ and $\tau$, so that for any $\epsilon>0$, $|z-t|_{C^0}(x)\leq C|x|^{1-\tau+\epsilon}$ as $|x|\to +\infty$.
 \end{lemma}
\begin{proof}

 By the paragraph after the proof of Proposition \ref{pro: entire graph}, $(\Sigma'_t, g|_{\Sigma'_t})$ can be regarded as $(\mathbf{R}^{n-1}, ds^2)$ with $ds^2$ satisfying
 \[
 \frac{1}{C_0}\delta_{ij}\leq ds^2\leq C_0\delta_{ij}
 \]
 for some uniform $C_0$.  Thus,  \eqref{eq: mse} can be written as 
 \begin{align}\label{eq: 23}
        \left\{
		\begin{alignedat}{2}
			& \Delta_G u=f \quad \text {in $\mathbf{R}^{n-1}$} \\
		   & u(x)\to 0  \quad\text {as $x\to \infty$} 
		\end{alignedat}
	\right.           
        \end{align}
Here $\Delta_G$ denotes the Laplacian operator on  $(\mathbf{R}^{n-1}, ds^2)$ and $(G_{ij})$ is uniformly elliptic on $\mathbf{R}^{n-1}$. To demonstrate that  $|z-t|$ decay to zero at  the  rate specified in the Lemma, we still let $u=z(x)-t$ as before. Since  the smallest and the largest eigenvalue of $(G_{ij})$ are uniformly bounded, then  By \cite{LSW63} ( c.f. (7.9), p.67), we know that the Green function $G(x,y)$ for $-\Delta_G$ satisfies 

\begin{equation} \label{eq:Greenfun}
C_1|x-y|^{3-n}\leq G(x,y)\leq 	C_2|x-y|^{3-n}, \quad \text{for any $x,y$ in $\mathbf{R}^{n-1}$,}
\end{equation}
here $C_1$ and  $C_2$ are two constants depends only on $n$ and $C_0$.
 Now, we have
\begin{equation}
	u(x)=-\int_{\mathbf{R}^{n-1}}G(x,y)f(y) dy, 
\nonumber
\end{equation}	
and

\begin{equation}\label{eq:24}
\begin{split}
|u(x)|&\leq C\int_{\mathbf{R}^{n-1}}|x-y|^{3-n}(1+|y|)^{-\tau-1} dy\\
&\leq C \left(\int_{B_{\frac r2}(o)}|x-y|^{3-n}(1+|y|)^{-\tau-1} dy+\int_{\mathbf{R}^{n-1}\setminus (B_{\frac r2}(o)\cup B_{3r}(x) )}|x-y|^{3-n}(1+|y|)^{-\tau-1} dy\right. \\
 &+ \left. \int_{ B_{3r}(x) \setminus B_{\frac r2}(o)}|x-y|^{3-n}(1+|y|)^{-\tau-1} dy \right),
\end{split}
\end{equation}		
where $r=\frac{|x|}{2}$ and we have used the fact
$$
|f|(y)\leq C(1+|y|)^{-\tau-1}
$$	
in the first inequality in \eqref{eq:24}, here and in the sequel  $C$ denotes a constant depends only on $n$ and its meaning is different from line to line .

For the case that $\tau<n-2$  we have
\begin{equation}\label{eq:25}
\begin{split}
\int_{B_{\frac r2}(o)}|x-y|^{3-n}(1+|y|)^{-\tau-1} dy&\leq C r^{3-n}\int_{B_{\frac r2}(o)}(1+|y|)^{-\tau-1} dy\\
&\leq C r^{3-n}\int^{\frac r2}_0(1+\rho)^{-\tau-1}\rho^{n-2}d\rho\\
&\leq C r^{1-\tau},
\end{split}
\end{equation}	
and if $\tau=n-2$ we have
$$
\int_{B_{\frac r2}(o)}|x-y|^{3-n}(1+|y|)^{-\tau-1} dy \leq C r^{3-n}\log r.
$$
Note that for any $y\in \mathbb{R}^{n-1}\setminus (B_{\frac r2}(o)\cup B_{3r}(x))$ there holds 
$$|x|\leq \frac{|y|}{2},$$ 
and hence
$$
|x-y|\geq \frac12 |y|,
$$
thus
\begin{equation}\label{eq:26}
\begin{split}
\int_{\mathbf{R}^{n-1}\setminus (B_{\frac r2}(o)\cup B_{3r}(x) )}|x-y|^{3-n}(1+|y|)^{-\tau-1} dy&\leq C\int_{\mathbf{R}^{n-1}\setminus B_{\frac r2}(o)}	|y|^{2-n-\tau}dy\\
&\leq C r^{1-\tau},
\end{split}
\end{equation}

\begin{equation}\label{eq:27}
\begin{split}
\int_{ B_{3r}(x) \setminus B_{\frac r2}(o)}|x-y|^{3-n}(1+|y|)^{-\tau-1} dy&\leq C r^{-\tau-1}\int_{ B_{3r}(x)}|x-y|^{3-n}\\
&\leq  C r^{1-\tau}\end{split}
 \end{equation}
Plug \eqref{eq:25}, \eqref{eq:26} and \eqref{eq:27} into \eqref{eq:24} we finally obtain that for any $\epsilon>0$ and $|y|$ large enough there holds

$$
|u|(y)\leq |y|^{1-\tau+\epsilon}.
$$
Thus we get the conclusion. 

\end{proof}

Given an area minimizing hypersurface $\Sigma'_t$ asymptotic to $S_t$ for some $t\geq t_0$, $\Sigma'_t$ can be represented as the entire graph on $S_t$ by $\Sigma'_t=\{(y, u), y\in S_t\}$, here and in the following, $u$ is a smooth function defined on $S_t$. To obtain uniform decay for higher order derivative of $u$, we need the following minimal graph equation(\cite{CM11} P236).
\begin{lemma}\label{eq: minimal graph}
Let $\Sigma_t=\{(y, u), y\in S_t\}$ be the area minimizing surface as above. Then
\begin{equation}\label{eq: minimal graph eq}
\begin{split}
     F(y,u(y),D_iu(y),D_{ij}^2u(y))&:= h^{ij}(D_{ij}u+\Gamma_{ij}^n+D_iu\Gamma_{nj}^n+D_ju\Gamma_{ni}^n+D_iuD_ju\Gamma_{nn}^n)\\
    &-D_kuh^{ij}(\Gamma_{ij}^k+D_iu\Gamma_{nj}^k+D_ju\Gamma_{in}^k+D_iuD_ju\Gamma_{nn}^k)=0,
\end{split}
\end{equation}
where $(h^{ij})$ is the inverse matrix to $(h_{ij})$ and 
\begin{align}\label{eq: elliptic coefficient}
    h_{ij}(y,u,p_i) = g_{ij}(y,u)+p_ig_{jn}(y,u)+p_jg_{in}(y,u)+p_ip_jg_{nn}(y,u)
\end{align}
and each Christoffel symbol is evaluated at $(y,u)$ with $1\leq i,j,k\leq n-1$.
\end{lemma}

\begin{lemma}\label{lem: decay estimate for higher derivative of u}
There is a constant $C$ depends only on $k$,  $\alpha$ and $(M^n,g)$ so that for any $k\geq 1$, $\epsilon>0$, $|u-t|_{C^{k,\alpha}}(y)\leq C|y|^{1-\tau-k+\epsilon}$ as $|y|\to +\infty$.
\end{lemma}

\begin{proof}
for which we could write as $h^{ij}D_{ij}u = f_1$. For $x\in \mathbf{R}^{n-1}$, denote $r = \frac{|x|}{4}$. By uniform Schauder estimate from the remark behind Lemma \ref{lem: L^p decay rsp t}, we have $|h_{ij}|_{\alpha}$ and $|D_iu|_{\alpha}$ uniformly bounded. It follows that $f_1 = O(|x|^{-\tau-1})$. Moreover, we have
\begin{equation}\label{eq: 38}
    \begin{split}
        \frac{\partial \Gamma_{ij}^k(x,u(x))}{\partial x_l} =
     \frac{\partial \Gamma_{ij}^k}{\partial x_l}+\frac{\partial \Gamma_{ij}^k}{\partial x_n}D_lu = O(|x|^{-2-\tau}).
    \end{split}
\end{equation}
This yields
\begin{equation}\label{eq: 30}
    \begin{split}
        &|\Gamma_{ij}^k|_{C^{\alpha}(B_r(x))}\le r^{1-\alpha}|D\Gamma_{ij}^k| = O(|x|^{\alpha-\tau-1})\\
    &|f_1|_{0,\alpha;B_r(x)} = O(|x|^{\alpha-\tau-1})
    \end{split}
\end{equation}
Hence, we obtain $|h_{ij}|_{C^{\alpha}(B_r(x))}+ r^2|f_1|_{C^{\alpha}(B_r(x))}<\Lambda$ for a uniform constant $\Lambda$ when $x$ is sufficiently large. By \cite{GT} Theorem 6.2, we deduce
\begin{align*}
    &r|Du(x)|+r^2|D^2u(x)|\\
    \le&|u|^*_{2,\alpha; B_r(x)} \le C(|u|_{0,B_r(x)}+r^2|f_1|_{0,\alpha;B_r(x)})\\
    =& O(|x|^{-\tau+1+\epsilon})
\end{align*}
Here $C$ relies only on $\alpha,(M^n,g)$ and $\Lambda$.
This gives the desired estimate for $|Du|$ and $|D^2u|$.
By taking derivative on two sides of \eqref{eq: minimal graph eq} and iterating the argument above, we get the desired decay estimate for higher order derivative of $u$.
\end{proof}

In the last of this subsection, we show the uniqueness result for area minimizing surface asymptotic to same hyperplane $S_t$. 

Consider two  area minimizing hypersurfaces asymptotic to same hyperplane $S_t$ for some $t\geq t_0$. 
Denote the graph function to be $u,v$ respectively. 
Then we have
\begin{equation}\label{eq: formal minimal surface eq}
    \begin{split}
        F(y,u(y),D_iu(y),D_{ij}^2u(y)) = 0\\
    F(y,v(y),D_iv(y),D_{ij}^2v(y)) = 0
    \end{split}
\end{equation}
 Denote $w = v-u$.
Then following P237 of \cite{CM11} we substract the two equations \eqref{eq: formal minimal surface eq} and get
\begin{equation}\label{eq: eliptic eq for w}
    \begin{split}
        &(\int_0^1 F_{p_{ij}}ds)D_{ij}w+(\int_0^1 F_{p_{i}}ds)D_iw+(\int_0^1 F_{u}ds)w\\
    =:& a_{ij}D_{ij}w+b_iD_iw+cw = 0
    \end{split}
\end{equation}
where each derivative of $F$ is evaluated at $(y,u+sw,D_iu+sD_iw,D_{ij}u+sD_{ij}w)$.
Since $F_{p_{ij}} = h^{ij}$, by the remark after Proposition \ref{pro: entire graph}, we have $a_{ij}$ is uniformly elliptic for $t\geq t_0$:
\begin{align*}
    \lambda |\xi|^2\le a_{ij}\xi^i\xi^j
\end{align*}
The next lemma presents decay estimate of coefficients in \eqref{eq: eliptic eq for w}
\begin{lemma}\label{lem: Estimate of Dh^{ij}, F_{p_i}, F_u}
On $S_t$, evaluating at $(y,u(y),Du(y),D^2u(y))$, there holds
\begin{align*}
    |Dh^{ij}|+|F_{p_i}| + |F_u|= O(|y|^{-\tau-1})
    \end{align*}
    and
    \begin{align*}
   ||Dh^{ij}||_{L^m}+ ||F_{p_i}||_{L^m}+||F_u||_{L^{\frac{m}{2}}}\to 0 (t\to \infty)
\end{align*}
\end{lemma}
\begin{proof}
 We first establish the decay estimate. We calculate
 \begin{equation}\label{eq: partial deriative of metric}
 \begin{split}
     D_{\theta}h^{ij}=&-h^{ik}h^{lj}D_{\theta}h_{kl}\\
     =&- h^{ik}h^{lj}D_{\theta}(g_{kl}+D_kug_{ln}+D_lug_{kn}+D_kuD_lug_{nn})
 \end{split}
 \end{equation}
 For the next two terms, we use $D_{p_{\theta}}$ and $D_u$ to denote partial derivative calculated with respect to $F$. Rewrite \eqref{eq: minimal graph eq} as $F = h^{ij} W_{ij}=0$. Then
    \begin{align*}
    F_{p_{\theta}}=&(D_{p_{\theta}}h^{ij})W_{ij}+h^{ij}D_{p_{\theta}}W_{ij}\\
         =& -h^{ik}h^{lj}D_{p_{\theta}}h_{kl}W_{ij}+h^{ij}D_{p_{\theta}}W_{ij}\\
        = & -h^{ik}h^{lj}(\delta_{k\theta}g_{ln}+\delta_{l\theta}g_{kn}+\delta_{k\theta}p_lg_{nn}+\delta_{l\theta}p_kg_{nn})W_{ij}+h^{ij}D_{p_{\theta}}W_{ij}
    \end{align*}
  Similarly,
\begin{align*}
    F_u=&(D_u h^{ij})W_{ij}+h^{ij}D_u W_{ij}\\
         =& -h^{ik}h^{lj}D_u h_{kl}W_{ij}+h^{ij}D_uW_{ij}\\
        = & -h^{ik}h^{lj}(p_k\frac{\partial g_{ln}}{\partial x_n}+p_l\frac{\partial g_{kn}}{\partial x_n} + p_kp_l\frac{\partial g_{nn}}{\partial x_n})W_{ij}+h^{ij}D_uW_{ij}
    \end{align*}
Then by the asymptotic flatness of $M$ and  Lemma \ref{lem: height estimate}, 
\[
|\partial g|+|Du||g_{kn}|+|D^2u|=O(|y|^{-\tau-1})
\]
Thus, we have $|Dh^{ij}|=O(|y|^{-\tau-1})$.
Similarly, $D_{ij}u=O(|y|^{-\tau-1})$ and any other term in $W_{ij}u$ consists the factors Christoffel symbol $\Gamma$,
it follows $W_{ij}=O(|y|^{-\tau-1})$. It's easy to show that $D_{p_{\theta}}W_{ij}$ and $D_uW_{ij}$ decays fast in order $-\tau-1$. Thus, we get the desired decay estimate.

To derive the integral estimate, we observe
\begin{equation}\label{estimate for W_{ij}}
\begin{split}
|Dh^{ij}|=&
    O(|\partial g|+|\partial g||Du|+|\partial g||Du|^{2}
    +|g_{kn}||D^2u|+|Du||D^2u|)\\
|W_{ij}|=&O(|Du|^2+|\partial g||Du|+|\partial g||Du|^2)\\
|D_{p_{\theta}}W_{ij}|=&O(|\partial g|+|\partial g||Du|)\\
\end{split}
\end{equation}
 By Lemma \ref{lem: L^p decay rsp t} and the remark after it, we have for some uniform $C$
 \begin{equation}\label{uniform bound for $Du, D^2u$}
|Du|+|D^2 u|\leq C.
 \end{equation}
 Then 
 \begin{align*}
    |Dh^{ij}|+|F_{p_{\theta}}|=&O(|\partial g|+|Du|^2+|g_{kn}||D^2u|+|Du||D^2u|)\\
    \leq &C(|\partial g|+|g_{kn}|^2+|Du|^2+|D^2u|^2)
 \end{align*}
Note for $q>m+1$,  we have
 \begin{align*}
        \int_{\mathbf{R}^m}\frac{1}{(|y|+t)^{q}}dy
        \leq C\int_0^{+\infty}\frac{\rho^{m-1}}{(\rho+t)^{q}}d\rho
        \leq C t^{m-q}\int_0^{+\infty}\frac{s^m}{(1+s)^{q}}ds\leq C_qt^{m-q}.
    \end{align*}
  Since  $|\partial g|\leq C(|y|+t)^{-1-\tau}$ and $m(1+\tau)\geq \frac{m}{2}(m+1)>m+1$, we get
  \begin{align}\label{integral estimate for g1}
    \int_{S_t}|\partial g|^md\mu\leq C  \int_{\mathbf{R}^m}\frac{1}{(|y|+t)^{m+m\tau}}dy
        \leq C t^{-m\tau}.
    \end{align}
Similarly, $|g_{kn}|\leq C(|y|+t)^{-\tau}$ and $2m\tau\geq m(m-1)>m+1$, then we have
 \begin{align*}
\int_{S_t}|g_{kn}|^{2m}d\mu\leq C  \int_{\mathbf{R}^m}\frac{1}{(|y|+t)^{2m\tau}}dy
        \leq C t^{m-2m\tau}.
    \end{align*}    
For $p\ge \frac{2m}{m-2}$, by Lemma \ref{lem: L^p decay rsp t}, we have
    \begin{align*}
        ||u-t||_{W^{2,p}}\leq C(t)\ \ \text{with}\quad C(t)\to 0\mbox{ as }t\to \infty
    \end{align*}
As $2m\geq \frac{2m}{m-2}$, then
\begin{equation}\label{L^m decay rsp t}
   \int_{S_t}(|Du|^m+|D^2u|^m)d\mu\leq C(t)  \ \ \text{with}\quad C(t)\to 0\mbox{ as }t\to \infty
\end{equation}
Thus,
\begin{equation}\label{L^m integral estimate for F}
   \int_{S_t}(|Dh^{ij}|^m+|F_{p_{\theta}}|^m)d\mu\leq 
   C(t)\ \ \text{with}\quad C(t)\to 0\mbox{ as }t\to \infty
\end{equation}
Similarly, we have
\[
|D_uW_{ij}|=O(|\partial^2 g|+|\partial^2 g||Du|+|\partial^2 g||Du|^2) \ \ \text{and}\ \ 
 |F_u|=O(|\partial g||W_{ij}|+|\partial g|^2+|\partial^2 g|)
\]
It follows
\begin{equation}\label{estimate for F_u}
    \begin{split}
|F_u|^{\frac{m}{2}}\leq& C(|\partial g|^{\frac{m}{2}}|W_{ij}|^{\frac{m}{2}}
+|\partial g|^m+|\partial^2 g|^{\frac{m}{2}})\\
\leq& C(|\partial g|^m+|W_{ij}|^m+|\partial^2 g|^{\frac{m}{2}})\\
\leq &C(|\partial g|^m+|Du|^{2m}+|\partial^2 g|^{\frac{m}{2}}),
    \end{split}
\end{equation}
where we have used \eqref{estimate for W_{ij}} and \eqref{uniform bound for $Du, D^2u$}
in the last inequality.
Using $|\partial^2 g|\leq C(|y|+t)^{-2-\tau}$ and $\frac{m}{2}(2+\tau)\geq \frac{m}{4}(m+3)>m+1$, we get
  \begin{align}\label{integral estimate for g2}
    \int_{S_t}|\partial^2 g|^\frac{m}{2}d\mu\leq C  \int_{\mathbf{R}^m}\frac{1}{(|y|+t)^{\frac{m}{2}(2+\tau)}}dy
        \leq C t^{-\frac{m}{4}(m-1)}.
    \end{align}
Plugging \eqref{integral estimate for g1},  \eqref{L^m decay rsp t} and \eqref{integral estimate for g2}
into \eqref{estimate for F_u} yields
\begin{equation}\label{integral estimate for F_u}
\int_{S_t}|F_u|^{\frac{m}{2}}d\mu\leq C(t) \ \ \text{with}\quad C(t)\to 0\mbox{ as }t\to \infty
\end{equation}
Combining \eqref{L^m integral estimate for F} and \eqref{integral estimate for F_u} gives the desired integral estimate.   
\end{proof}
Since $S_t$ enjoys very nice properties for $t\gg1$, we can use Corollary \ref{cor: uniqueness of solution} to 
conclude that \eqref{eq: eliptic eq for w} has only trivial solution $w=0$. Hence, we obtain
\begin{proposition}\label{prop: uniqueness}
    There exists some $t_0$ such that any two area minimizing surfaces asymptotic to the same hyperplane $S_t$ for some $t\geq t_0$ must coincide.
\end{proposition}

\subsection{The foliation structure for AF manifold}
For $t>t_0$, we denote the unique area minimizing hypersurface asymptotic to $S_t$ in Proposition \ref{prop: uniqueness} by $\Sigma_t$. The first thing we will do is to show that $\{\Sigma_t\}$ actually form a $C^0$ foliation for $t>t_0$
\begin{proposition}\label{prop: foliation beyond t_0}
    The region in $M$ beyond $\Sigma_{t_0}$ can be $C^0$ foliated by $\{\Sigma_t\}_{t>t_0}$.
\end{proposition}

First, we prove the following two lemmas.
\begin{lemma}\label{lem: Sigma t_i converges to Sigma_t}
    For a sequence of $\{t_i\}$ with $t_i\to t> t_0$, we have $\Sigma_{t_i}\to \Sigma_t$.
\end{lemma}
\begin{proof}
    Since the second fundamental form of $\Sigma_{t_i}$ is uniformly bounded
    and $\Sigma_{t_i}$ is asymptotic to $S_{t_i}$, 
    then by  Proposition \ref{pro: C^0 estimate of u}, we have $C^0$ estimate on $\Sigma_{t_i}$, therefore it converges to an area minimizing  hypersurface asymptotic to $S_t$. By uniqueness result in Proposition \ref{prop: uniqueness} this hypersurface must coincide with $\Sigma_t$.
\end{proof}
\begin{lemma}\label{lem: t_1 beyond t_2}
    For $t_1>t_2$, let $\Sigma_i=\Sigma_{t_i}$ be the area minimizing hypersurface asymptotic to $S_{t_i}$ for $i=1,2$. Then $\Sigma_{1}$ lies strictly beyond $\Sigma_{2}$.
\end{lemma}
\begin{proof}
    By asymptotic property (Proposition 9 in \cite{EK23}) $\Sigma_{1}$ lies strictly beyond $\Sigma_{2}$ outside a compact set. If these two minimal hypersurface intersect, then they must intersect transversally. Let $D_i$ be the compact region on $\Sigma_i$ bounded by the transversal intersection of $\Sigma_1$ and $\Sigma_2$. Consider $(\Sigma_1\backslash D_1)\cup D_2$, which turns out to be another area minimizing boundary. This violates the regularity theory of area minimizing hypersurface for dimension less than $7$.
\end{proof}

\begin{proof}[Proof of Proposition \ref{prop: foliation beyond t_0}]
    Denote $U$ to be the region beyond $\Sigma_{t_0}$. 
    It suffices to show
    for any point $p$ in $U$, there exists a hypersurface $\Sigma_t$ passing through $p$ for some $t>t_0$. Let $l$ be the portion of line passing through $p$ in $U$, paralleling to $z$-axis under Euclidean metric. Denote
    \begin{align*}
        \mathcal{S} = \lbrace q\in l | \exists \Sigma_t, \Sigma_t\cap l = q\rbrace
    \end{align*}
    By Lemma \ref{lem: Sigma t_i converges to Sigma_t} $\mathcal{S}$ must be close. If $p$ is not in $l$, then we can find $p_1,p_2\in l$, such that
    \begin{equation}\label{eq: 19}
        \begin{split}
            &z(p_1) = \inf_{q\in l, z(q)>z(p)}z(q)\\
        &z(p_2) = \sup_{q\in l, z(q)<z(p)}z(q)
        \end{split}
    \end{equation}
    Let $p_i\in\Sigma_{t_i}$. By graph property we have $t_1\ge t_2$. If $t_1>t_2$, then for some $t_3\in (t_1,t_2)$, $p_3 = \Sigma_{t_3}\cap l$ lies between $p_1$ and $p_2$, this contradicts with \eqref{eq: 19}. Hence $t_1=t_2$, and this contradicts with uniqueness Proposition \ref{prop: uniqueness}.
\end{proof}

 Next, we improve the foliation $\lbrace\Sigma_t\rbrace_{t>t_0}$ into a $C^1$ foliation.

\begin{proposition}\label{prop: C^1 smoothness for foliation}
    The foliation $\lbrace\Sigma_t\rbrace_{t>t_0}$ is $C^1$ in $t$.
\end{proposition}
\begin{proof}
Let $u_t$ and $u_{t+\delta}$ be the graph function of  
$\Sigma_t$ and $\Sigma_{t+\delta}$  over $S_t$ respectively. Then 
\begin{equation}\label{eq: 33}
    \begin{split}
        F(y,u_t(y),D_iu_t(y),D_{ij}^2u_t(y)) = 0\\
    F(y,u_{t+\delta}(y),D_iu_{t+\delta}(y),D_{ij}^2u_{t+\delta}(y)) = 0
    \end{split}
\end{equation}
By Lemma \ref{lem: decay estimate of u}, we have
\begin{align*}
    &|u_t-t|+|y||Du_t|+|y|^2|D^2u_t| = O(|y|^{1-\tau+\epsilon})\\
    &|u_{t+\delta}-(t+\delta)|+|y||Du_{t+\delta}|+|y|^2|D^2u_{t+\delta}| = O(|y|^{1-\tau+\epsilon})
\end{align*}
 Denote $w_{\delta} = \frac{u_{t+\delta}-u_t}{\delta}-1$. By substracting two equations in \eqref{eq: 33}, we obtain
\begin{align*}
    (\int_0^1 F_{p_{ij}}ds)D_{ij}w_{\delta}+(\int_0^1 F_{p_{i}}ds)D_iw_{\delta}+(\int_0^1 F_{u}ds)w_{\delta} = -(\int_0^1 F_{u}ds)
\end{align*}
where each derivative of $F$ is evaluated at $(y,u_{\delta}+sw_{\delta},D_iu_{\delta}+sD_iw_{\delta},D_{ij}u_{\delta}+sD_{ij}w_{\delta})$. We rewrite the above equation as
\begin{align}\label{eq: 34}
    L_{\delta}w_{\delta}= a_{ij}^{\delta}D_{ij}w_{\delta}+b_i^{\delta}D_iw_{\delta}+c^{\delta}w_{\delta} = -c^{\delta}
\end{align}
The following estimate follows easily from the argument of Lemma \ref{lem: Estimate of Dh^{ij}, F_{p_i}, F_u}:
\begin{align}\label{eq: 42}
    |Da_{ij}^{\delta}|+|b_i^{\delta}|+ |c^{\delta}|& = O(|y|^{-\tau-1})\\
 ||Da_{ij}^{\delta}||_{L^m}+||b_i^{\delta}||_{L^m}+||c^{\delta}||_{L^{\frac{m}{2}}}
 &\to 0 (t\to \infty)
\end{align}
By Lemma \ref{lem: decay estimate of u}, we see $|w_{\delta}|\leq \frac{C}{\delta}(|x|^{-\tau+1+\epsilon})$, hence $L^{\frac{2m}{m-2}}$ integrable. Combining the $L^p$ estimate for coefficients \eqref{eq: 42} with \eqref{eq: L^p bound} in Appendix we obtain uniform $L^{\frac{2m}{m-2}}$ estimate for $w_{\delta}$. By Moser iteration $w_{\delta}$ is uniformly bounded (\cite{GT}, Theorem 8.17), so $L^p$ norm of $w_{\delta}$ is uniformly bounded for all $p>\frac{2m}{m-2}$. Combining with $W^{2,p}$ estimate and Schauder estimate, we know the $C^{k,\alpha}$ norm of $w_{\delta}$ is uniformly bounded. Consequently, there exists a subsequence $w_{\delta_i}$ converging to a function $w$ with respect to $C^{k,\alpha}$ norm with  $w$  satisfying the following equation:
\begin{align}\label{eq: 43}
    h^{ij}D_{ij}w+F_{p_i}D_iw+F_uw = -F_u
\end{align}
Moreover, it follows from \eqref{eq: 42} and the uniform $C^{1,\alpha}$ estimate of $w_{\delta}$ that if we write
\begin{align*}
    D_i(a_{ij}^{\delta}D_jw_{\delta}) = f_{\delta}
\end{align*}
then $f_{\delta} = O(|y|^{-\tau-1})$. Similar to the argument of Lemma \ref{lem: decay estimate of u}, we get the following inequality holds for a constant $C$ independent of the choice of $\delta$:
\begin{align*}
    |w_{\delta}|+|y||Dw_{\delta}|+|y|^2|D^2w_{\delta}| \le C|y|^{1-\tau+\epsilon}
\end{align*}
Therefore, the same thing also holds for $w$:
\begin{align}\label{eq: 44}
    |w|+|y||Dw|+|y|^2|D^2w| \le C|y|^{1-\tau+\epsilon}
\end{align}
By Lemma \ref{lem: L^ 2m/m-2 estimate}, the solution of \eqref{eq: 43} satisfying \eqref{eq: 44} is unique. Combining with uniform Schauder estimate for $w_{\delta}$ we conclude that the limit of $w_{\delta}$ as $\delta\to 0$ exists and equals $w$. By uniform boundness of $w_{\delta}$ we see $u$ is Lipschitz in $t$. 

To see $u$ is $C^1$ in $t$, we denote $u^{(1)}_t = w+1$ above with respect to $\Sigma_t$. \eqref{eq: 43} then interprets as
\begin{align*}
    h^{ij}_tD_{ij}u^{(1)}_t+F_{p_i}^tD_iu^{(1)}_t+F_u^tu^{(1)}_t = 0
\end{align*}
By substracting the equations respect to $t$ and $t+\delta$, we obtain
\begin{align*}
    &h^{ij}_tD_{ij}(u^{(1)}_{t+\delta}-u^{(1)}_t)+F_{p_i}^tD_i(u^{(1)}_{t+\delta}-u^{(1)}_t)+F_u^t(u^{(1)}_{t+\delta}-u^{(1)}_t) \\
    = &(h^{ij}_t-h^{ij}_{t+\delta})D_{ij}u^{(1)}_{t+\delta}+(F_{p_i}^t - F_{p_i}^{t+\delta})D_iu^{(1)}_{t+\delta}+(F_u^t-F_u^{t+\delta})u^{(1)}_{t+\delta}
\end{align*}
By \eqref{eq: 42}\eqref{eq: 44}, the coefficients in the above equation has fast decay, so we can apply Lemma \ref{lem: L^ 2m/m-2 estimate} to obtain a $L^{\frac{2m}{m-2}}$ estimate for $|u^{(1)}_{t+\delta}-u^{(1)}_t|$. Standard iteration improves this to a $C^0$ bound. Consequently, $|u^{(1)}_{t+\delta}-u^{(1)}_t|$ tends to zero as $\delta\to 0$. This concludes the proof of Proposition \ref{prop: C^1 smoothness for foliation}.
\end{proof}

By the similar arguments, we can demonstrate that the foliation $\lbrace\Sigma_t\rbrace_{t>t_0}$ is $C^k$ in $t$ for all $k\ge 1$. Similarly, the same thing also holds when $t$ is sufficiently negative. As a result, we can find $T_0>0$, such that the regions in $M$ beyond $\Sigma_{T_0}$ and below $\Sigma_{-T_0}$ coincide with the smooth foliation $\lbrace\Sigma_t\rbrace_{t>T_0}$ and $\lbrace\Sigma_t\rbrace_{t<-T_0}$. Next we're interested in the characterization of the region between $\Sigma_{-T_0}$ and $\Sigma_{T_0}$.

 \begin{figure}
    \centering
    \includegraphics[width = 15cm]{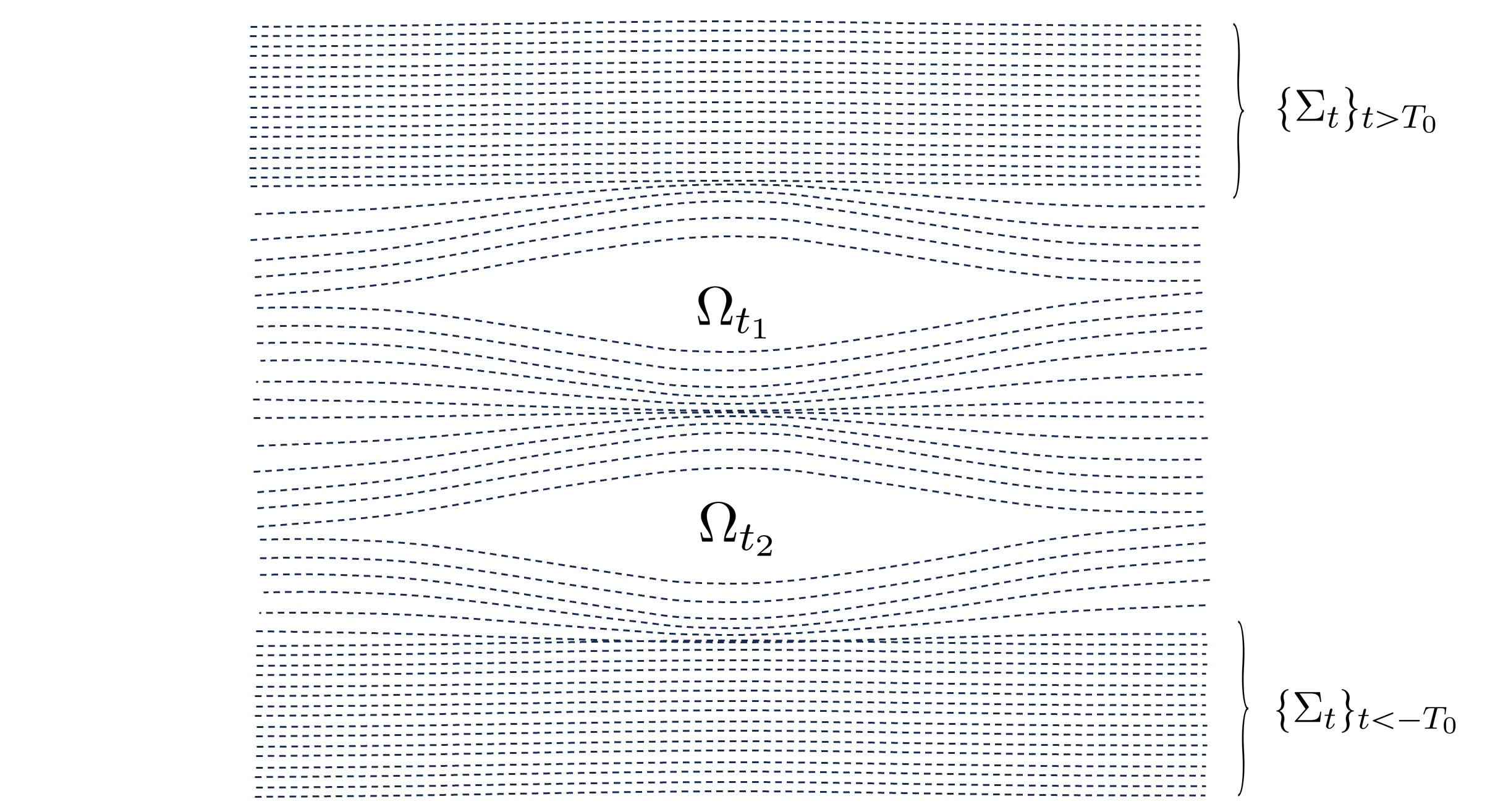}
    \caption{The foliation structure of $M$}
    \label{f3}
\end{figure}    

    For $t\in\mathbb{R}$, we define $\mathcal{A}_t$ to be the collection of area minimizing boundaries $\Sigma$ asymptotic to $S_t$ at infinity. Pick a sequence $t_i\to t^+$. By Lemma \ref{lem: decay estimate of u}, there exists $\Sigma_i\in \mathcal{A}_{t_i}$. For a fixed element $\Sigma\in \mathcal{A}_t$, we have $\Sigma$ strictly below $\Sigma_i$ for each $i$ according to Lemma \ref{lem: t_1 beyond t_2}. Let $i\to\infty$, by uniform $C^0$ decay estimate of Lemma \ref{lem: decay estimate of u} $\Sigma_i$ must converge to an element in $\mathcal{A}_t$, denoted by $\Sigma^+$. Then $\Sigma$ either coincide with $\Sigma^+$, or lies strictly below $\Sigma^+$. By exactly the same procedure we are able to construct an element $\Sigma_-$ in $\mathcal{A}_t$. The discussion above summarizes to the following lemma:

    \begin{lemma}
        For any $t\in\mathbb{R}$, one of the following two cases happens:

        (1) $|\mathcal{A}_t| = 1$. In this case area minimizing boundary aymptotic to $S_t$ is unique.

        (2) $|\mathcal{A}_t| > 1$. In this case there exist $\Sigma_t^+,\Sigma_t^-\in \mathcal{A}_t$ which bound a region $\Omega_t$, such that all other element in $\mathcal{A}_t$ lies in $\Omega_t$.

        Here, $|\mathcal{A}_t|$ denotes the number of elements of $|\mathcal{A}_t|$ (not necessarily finite).
    \end{lemma}

    Summarizing the results above, we obtain the following structure theorem for asymptotic manifold of dimension $4\le n\le 7$

    \begin{proposition}
        For any asymptotic flat manifold $(M^n,g)$ satisfying $4\le n\le 7, \tau>\frac{n-1}{2}$, there exists $T_0>0$, such that

        (1) For $|t|>T_0$, the area minimizing boundary in $M$ asymptotic to $S_t$ is unique. Moreover, the region outside $\Sigma_{\pm T_0}$ is smooth foliation by $\lbrace\Sigma_t\rbrace_{|t|> T_0}$.

        (2) For $|t|\le T_0$, denote
        \begin{align*}
            &I = \lbrace t,\quad |\mathcal{A}_t|>1\rbrace 
        \end{align*}
        Then $I$ is countable, and the region between $\Sigma_{\pm T_0}$ could be represented as
        \begin{align*}
            (\bigsqcup_{\alpha\in I}\Omega_{\alpha})\cup(\bigsqcup_{\beta\in [-T_0, T_0]\backslash I}\Sigma_{\beta})
        \end{align*}
    \end{proposition}
    \begin{proof}
        By our definition, the area minimizing hypersurface asymptotic to $S_t$ is unique if and only if $t\in \mathbb{R}\backslash I$. For the remaining it suffice to show $I$ is countable. This follows immediately from the 1-1 correspondence between element $t\in I$ and the open set $\Omega_t\subset M$, and that these $\Omega_t$'s are mutually disjoint.
    \end{proof}

 \section{On the Schoen's conjecture}
	In this section, we give the proof of Theorem \ref{thm: free boundary}. Let $(M^n,g)$ be the asymptotically flat manifold and $C_{R_i}$ the cylinder as stated in Theorem \ref{thm: free boundary}, each of which contains a hypersurface $\Sigma_i$ minimizing the volume in $\mathcal{F}_{R_i}$. Assume the statement of the theorem is not true, then we are able to find a compact set $\Omega_0$ with $\Sigma_i\cap\Omega_0\ne\emptyset$ for each $i$. By passing to a subsequence, $\Sigma_i$ converges locally smoothly to an area minimizing boundary $\Sigma$. To  investigate the basic properties for $\Sigma$, we introduce the {\it height} of a subset $A\subset M$ as follows: the height of $A$ is defined to be
		\begin{align}\label{eq: 58}
			Z(A) = &Z_+(A)-Z_-(A)
		\end{align}
  with
  \begin{align*}
      Z_+(A) =\left\{
	 \begin{alignedat}{2}
			& \sup \lbrace z(p), p\in A\backslash K\rbrace, \quad A\backslash K\ne\emptyset,  \\
		   & 0,  \quad A\backslash K=\emptyset
		\end{alignedat}
	\right.\\
    Z_-(A) =\left\{
	 \begin{alignedat}{2}
			& \inf \lbrace z(p), p\in A\backslash K\rbrace, \quad A\backslash K\ne\emptyset,  \\
		   & 0,  \quad A\backslash K=\emptyset
     \end{alignedat}
	\right.
  \end{align*}

\begin{lemma}\label{lem: Sigma parallel to S0}
Let $\Sigma$ be described as above, then $\Sigma$ is a graph over $S_0$ outside a compact set, and the graph function $u$ satisfies
        \begin{align*}
			|y|_{\bar{g}}^{-1}|u(y)-a|+|\Bar{\nabla}u(y)|_{\Bar{g}}+|y|_{\bar{g}}|\Bar{\nabla}^2u(y)|=O(|y|^{-\tau+\epsilon})
		\end{align*}
	\end{lemma}
	\begin{proof}
     We claim the following holds
     \begin{align}\label{eq: 57}
         Z(\Sigma_i) = o(R_i)
     \end{align}
In fact, let $\Tilde{C}_{R_i} = \frac{1}{R_i}(C_{R_i}\backslash K)$ be the rescaled cylinder with radius 1, then $\Tilde{C}_{R_i}$ converges to $(C_1\backslash\{O\},g_{Euc})$ in Gromov-Hausdorff sense, with $C_1$ denoting the Euclidean cylinder with radius $1$. Since $\Sigma_i$ intersects with a compact set, we have $\frac{1}{R_i}\Sigma_i$ converge to a stable hypercone $\Sigma'$ in $C_1\backslash O$ with Euclidean metric, and $\partial\Sigma'\subset C_1$. From our assumption that $n\le 7$, it follows from the classification of stable minimal cone by J.Simons \cite{Simon1968} that $\Sigma'$ is actually a hyperplane. Moreover, as each $\Sigma_i$ minimizes the volume in $C_{R_i}$, we see $\Sigma'\cup O$ minimizes the volume in $C_1$. Hence, $\Sigma'$ coincides with the standard disk $\{z=0\}\cap C_1$, and \eqref{eq: 57} immediately follows. 

    Furthermore, by the minimizing property of each $\Sigma_i$ and $\Sigma'$, we see the convergence from $\frac{1}{R_i}\Sigma_i$ to $\Sigma'$ is of multiplicity one. From this it is not hard to see $\Sigma$, the limit of $\Sigma_i$, has exactly one end.

     Let us return to the proof of Lemma \ref{lem: Sigma parallel to S0}. By Proposition 9 in \cite{EK23}, $\Sigma$ is asymptotic to a hyperplane $\Pi$ at infinity, with desired decay estimate for derivatives. \eqref{eq: 57} implies $\Pi$ must be parallel to $S_0$, and the conclusion follows.
	\end{proof}

For convenience of later discussion, we introduce the following conception of slope
        \begin{definition}\label{def: slope}
            For a piecewise smooth hypersurface $V$ in $M\backslash K$, we define the slope of a regular point $p\in V$ to be the dihedral angle of the tangent hyperplane of $V$ at $p$ and $S_0$, with respect to Euclidean metric.
        \end{definition}
        Obviously, if we view $V$ as a local graph over $S_0$ with graph function $u$, then the slope at $p$ coincides with $\arctan|\Bar{\nabla}u|$.
 
	Clearly, $\Sigma$ seperates $M$ into two parts: the upper part $M_+$ and the lower part $M_-$. By Theorem \ref{thm: Existence}, there exists a positive number $t_0$, such that the region in $M$ beyond $\Sigma_{t_0}$ is smoothly foliated by $\Sigma_t (t\ge t_0)$. The key step in proving is to establish Lemma \ref{lem: coincide}, which claims the existence of $t_1>t_0$, such that for $t>t_1$, $\Sigma_t$ is stable under asymptotic constant variation. We would achieve this by picking a suitable point $p\in\Sigma_t$ and constructing a minimal hypersurface $\Sigma_p$ with desired property, and showing that $\Sigma_t$ coincides with $\Sigma_p$.
 
 For $t>t_0$, we can pick a point $p\in \Sigma_t$ with sufficiently large $x_1$-coordinate, such that the minimizing geodesic $\gamma$ joining $p$ and $\Sigma$ is very far from the compact set $K$ defined in Definition \ref{def: AF manifold}, and $\gamma$ is almost perpendicular
        \begin{equation}\label{eq: 63}
            \begin{split}
                &\inf\{x_1(q),q\in\gamma\}>100\\
                &\langle \gamma',\frac{\partial}{\partial z}\rangle>1-\frac{1}{10000}
            \end{split}
        \end{equation}
 This is always possible due to the asymptotic flatness of $M$ and Lemma \ref{lem: Sigma parallel to S0}.

 We present the following metric deformation lemma, which is a slight modification of Lemma 31 from \cite{EK23}.
        \begin{lemma}\label{lem: g(s) perturbation}
            Let $p\in M_+$ be stated as above, then for any $r_0>0$, there exists $0<r<r_0$, an open set $W\subset M$ with compact closure, satisfying $W\cap \Sigma\ne\emptyset$, and a family of Riemannian metrics $\lbrace g(s)\rbrace_{s\in[0,1]}$, such that the following holds

            (1) $g(s)\to g$ smoothly as $t\to 0$

            (2) $g(s)=g$ in $M\backslash W$

            (3) $g(s)<g$ in $W$

            (4) $R(g(s))>0$ in $\lbrace x\in W: \dist(x,p)>r\rbrace$

            (5) For $i$ sufficiently large, $\Sigma_i$ is weakly mean convex and strictly mean convex at one interior point under $g(s)$ with respect to the normal vector pointing into $M_-$.

            Moreover, $W$ satisfies following properties:
            
            (I) $\partial W$ is piecewise smooth.

            (II) If there exists a smooth hypersurface $V$ touching $\partial W$ from the outside at $q$, then $q$ lies in the regular part in $\partial W$.
            
            (III) There exists a constant $\epsilon_0>0$, if $q\in W\cap M_+$ has slope smaller than $\epsilon_0$, then $\dist(q,p)<6r$.
        \end{lemma}
        \begin{proof}
            We follow the main strategy of Lemma 31 in \cite{EK23}. See also \cite{CCE16} and \cite{Li24}.

            By the calculation of \cite{Li24}, if we can find $v\in C^{\infty}(M)$ and the set $W$, such that
            \begin{equation}\label{eq: 50}
                \begin{split}
                    &  v = 0 \mbox{ in }M\backslash W\\
                &  v < 0 \mbox{ in }W\\
                &  \Delta v<0 \mbox{ in }\lbrace x\in W: \dist(x,p)>r\rbrace\\
                &  \frac{\partial v}{\partial\nu_i}>0 \mbox{ on }\Sigma_i
                \end{split}
            \end{equation}
            Then $g_s = (1+sv)^{\frac{4}{n-2}}g$ satisfies the desired property.

            As in \cite{Li24} we can find a nonpositive function $f\in C^{\infty}(\mathbb{R})$, such that
            \begin{align*}
                &f'(x)>0\mbox{ on }(0,6)\\
                &(4n-1)f'(x)+xf''(x)<0\mbox{ on }(1,6)\\
                &f(x)=0\mbox{ for }x\ge 6
            \end{align*}
            Moreover, by choosing $0<r_0<inj(M)$ sufficiently small, we can guarantee $\Delta_g d(x,q)^2<8n$ on any geodesic ball $B_{6r}(q)$ for $r<r_0$. Define $v_{r,q} = r_0^2f(\frac{d(x,q)}{r_0})$, so the calculation in \cite{Li24} yields
            \begin{align*}
                &v_{r,q} = 0 \mbox{ in } M\backslash B_{6r}(q)\\
                &v_{r,q} < 0 \mbox{ in } B_{6r}(q)\\
                &\Delta v_{r,q} < 0 \mbox{ in }B_{6r}(q)\backslash B_{r}(q)
            \end{align*}
            Additionally, If $B_{6r}(q)\cap V\ne\emptyset$ for some hypersurface $V$, then there holds
            \begin{align}\label{eq: 51}
                \frac{\partial v_{r,q}}{\partial\nu} = r_0f'(\frac{d(x,q)}{r_0})\langle \nabla d(x,q),\nu\rangle \mbox{ on }V
            \end{align}
            with $\nu$ denoting the unit normal of $V$.

            Let $\gamma$ be the minimizing geodesic joining $p$ and $\Sigma$, with endpoints $p$ and $p_0\in\Sigma$. By our assumption $\gamma$ lies far from the compact set $K$, the geometry around $\gamma$ is very similar to that of the Euclidean space. Let $N$ be a large interger with $r=\frac{L(\gamma)}{4N}<r_0$. Choose $p_i(i=1,2,\dots,N)$ along $\gamma$ such that $\dist(p_i,p_{i+1}) = 4r (i=0,1,\dots,N-1)$ with $p_N=p$. Then $p_i$ coincides with the image of exponential map determined by $p_0$, by evaluating at $4ir$.

 \begin{figure}
    \centering
    \includegraphics[width = 15cm]{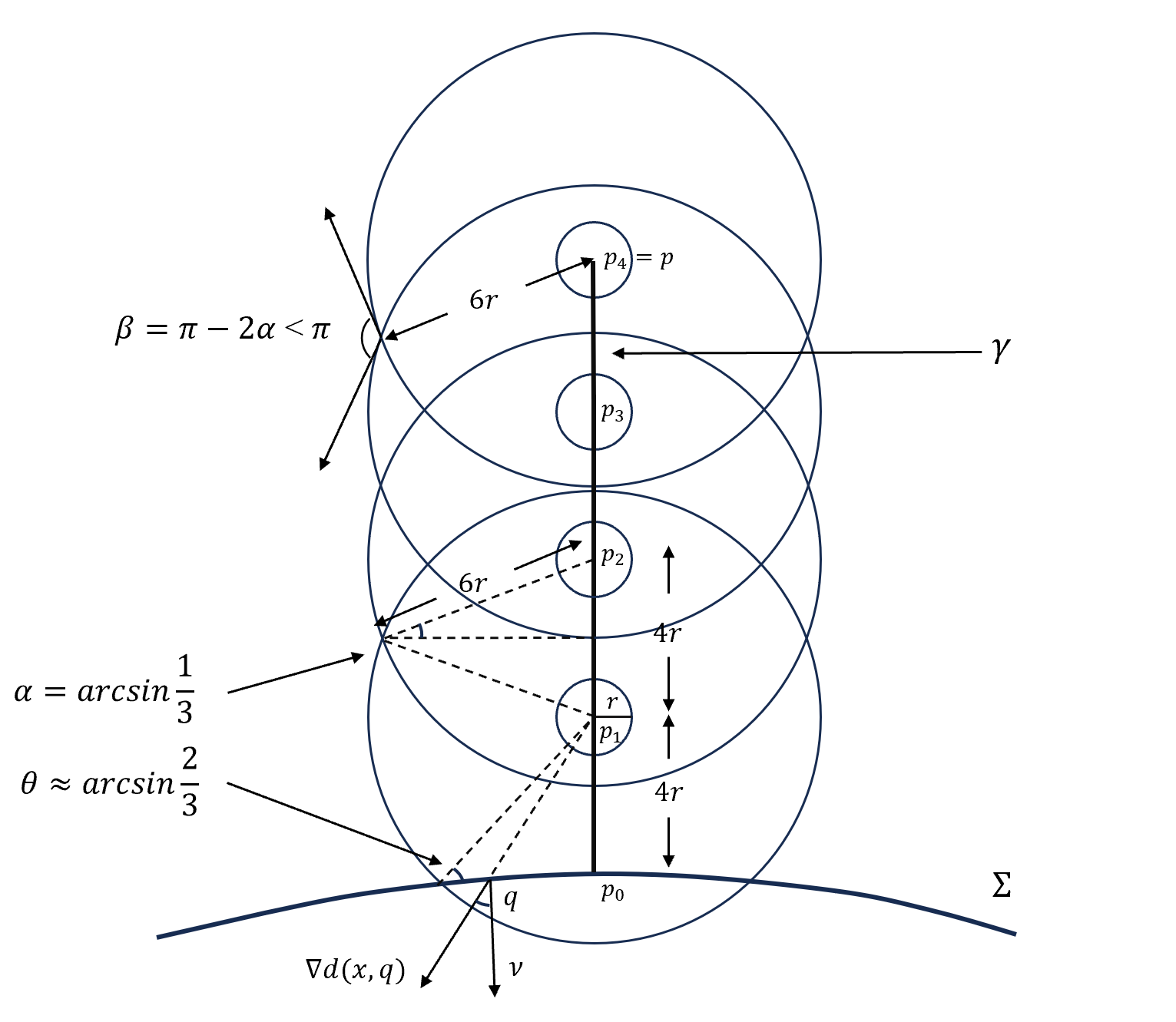}
    \caption{The construction of $W$}
    \label{f1}
\end{figure}           
It is readily to check that, the $p_i$'s chosen above satisfies
           \begin{align*}
               &B_r(p_i)\subset B_{5r}(p_{i+1})\backslash B_{3r}(p_{i+1}),\mbox{ for }i=1,2,\dots,N-1\\
               &B_{6r}(p_i)\cap\Sigma = \emptyset, i = 2,\dots,N; B_{6r}(p_1)\cap\Sigma\ne\emptyset.
           \end{align*}
           Following \cite{EK23}, we define $a_1,a_2,\dots,a_N$ recursively by $a_1 = 1$ and set $a_2,\dots,a_N$ by
           \begin{align*}
               a_i = 1+\frac{\sup\{\Delta v_{r,q_{i-1}}(x),x\in B_r(q_{i-1})}{\inf\{\Delta v_{r,q_i}(x),x\in B_{5r}(q_i)\backslash B_{3r}(q_i)}a_{i-1}
           \end{align*}
           Let
           \begin{align*}
               &v = \sum_{i=1}^N a_iv_{r,p_i}\\
               &W = \bigcup_{i=1}^N B_{6r}(p_i)\\
           \end{align*}
           then the first three equations in \eqref{eq: 50} could be satisfied.
           
           For the last equation in \eqref{eq: 50}, if we choose $r$ sufficiently small, then $\Sigma\cap B_{6r}(p_1)$ could be regarded as a hyperplane. It is not hard to see from Figure \ref{f1} that
           \begin{align}\label{eq: 55}
               \langle \nabla d(x,q),\nu\rangle>\frac{2}{3}-\epsilon'
           \end{align}
           for any $q\in\Sigma\cap B_{6r}(p_1)$, with $\nu$ denoting the outer normal of $\Sigma$ pointing inward $M_-$, and $\epsilon'$ can be arranged arbitrarily small by possibly choosing $p$ with larger $x_1(p)$. Since $\Sigma_i$ converges to $\Sigma$, \eqref{eq: 55} also holds for $\Sigma_i$ when $i$ is sufficiently large. The last equation of \eqref{eq: 50} then follows immediately \eqref{eq: 51}. This shows the metric $g(s)$ has desired properties.

            Next, we investigate properties (I)(II)(III) in describing the shape of $W$. (I) follows immediately from the construction. (II) follows from Figure \ref{f1} as the outer dihedral angle $\beta$ of tangent hyperplanes of $\partial B_{6r}(p_i)$ and $\partial B_{6r}(p_{i+1})$ on $\partial B_{6r}(p_i)\cap \partial B_{6r}(p_{i+1})$ is less than $\pi$. (III) also follows from Figure \ref{f1}, since for any $q\in \partial W\cap M_+$ the slope of $q$ with respect to $\partial W$ is bounded from below by $\cot\alpha>2\sqrt{2}-\epsilon'$ for $\epsilon'$ arbitrarily small, provided that $q$ is not in the upper hemisphere of $\partial B_{6r}(p)$, so the conclusion follows.

        \end{proof}
	
	Under this perturbed metric, we next aim to find a free boundary minimal surface in $\mathcal{F}_{R}$. We need the following existence result for free boundary minimal hypersurface.
	
	\begin{lemma} \label{lem: free boundary existence}
		Let $N$ be a $n$-Riemannian manifold with piecewise smooth boundary $(n\le 7)$, $\partial N=\partial_{eff}N\cup \partial_{side} N$, with $\partial_{eff}N=\partial_+N\sqcup \partial_-N$, such that each of $\partial_+N, \partial_-N$ and $ \partial_{side} N$ is smooth. Assume the mean curvature on $\partial_{\pm}N$ with respect to outer normal is everywhere nonnegative, and strictly positive somewhere on both $\partial_{\pm}N$. Furthermore, assume the dihedral angle of $\partial_+N$ with $\partial_{side}N$ is acute, and the dihedral angle of $\partial_-N$ with $\partial_{side}N$ equals $\frac{\pi}{2}$.
		
		Let 
		\begin{align*}
		    \mathcal{C} = \lbrace \Sigma = \partial\Omega\cap\mathring{N}, \Omega \mbox{ is a Caccioppoli set in } N, \partial_-N\subset\Omega, \partial_+N\cap\Omega = \emptyset \rbrace
		\end{align*}
		
		Then one can find a hypersurface $\hat{\Sigma}$ minimizing the volume in the class $\mathcal{C}$. Further more, $\hat{\Sigma}$ is smooth away from $\partial N$, and is disjoint with $\partial_{\pm}N$. 
	\end{lemma}
	
	\begin{proof}
		The lemma follows from Theorem 3.1 in \cite{Li24}. The acute angle part follows from Lemma 2.14 in \cite{CCZ23}, by letting $\mu = 0$ in Lemma 2.14 in \cite{CCZ23}.
	\end{proof}
	
	We need the following monotonicity formula for free boundary minimal hypersurface
	\begin{lemma}[Theorem 3.4 in \cite{Zhou20}]\label{lem: monotonicity}
		Let $(N,\partial N)$ be a Riemannian manifold with $|sec|<k$, injective radius bounded from below by $i_0$ and the square of the second fundamental form $|A|^2$ of $\partial N$ bounded from above by $\Lambda$. Let $(\Sigma,\partial\Sigma)\subset(N,\partial N)$ be a free boundary minimal hypersurface in $N$. Then there exists $L_0, \Lambda_0$ depending only on $k,i_0,\Lambda$ such that the following holds:
		\begin{align}\label{eq: monotonicity}
			e^{\Lambda_0\sigma}\frac{\mathcal{H}^{n-1}(B_{\sigma}(p))}{\sigma^{n-1}}\le e^{\Lambda_0\rho}\frac{\mathcal{H}^{n-1}(B_{\rho}(p))}{\rho^{n-1}}
		\end{align}
		for $0<\sigma<\rho<L_0$, when $p\in \partial N$.
	\end{lemma}
	
	Next, we study the existence of the area minimizing free boundary hypersurface in $(C_R,g(s))$. First, we need some notations: For $t_1,t_2$ with $|t_i|>1$ ($i = 1,2$), $t_1<t_2$, we denote $S_{[t_1,t_2]}$ to be the region in $M$ between $t_1$ and $t_2$. For a point $p\in M\backslash K$, we define $z(p)$ to be the $z$-coordinate of $p$. 

  We begin by showing the following result concerning height estimate of the free boundary hypersurface in cylinder.
        \begin{lemma}\label{lem: height estimate}
             There exists constants $R_0>0$, $\eta = \eta(k,i_0,\Lambda)>0$, with $k,i_0$ representing the curvature upper bound and injective radius lower bound of the asymptotically flat manifold $(M,g(s))$, $\Lambda$ representing the upper bound of second fundamental form of $C_R$, such that for $R>R_0$, there holds
			\begin{align}\label{eq: 52}
				\mathcal{H}^{n-1}(V_0\cap (C_R\cap S_{[-h,h]}))\ge \eta (Z(V_0\cap (C_R\cap S_{[-h,h]}))-2)
			\end{align}
   with $Z$ given by \eqref{eq: 58}. Here, $V_0$ is an embedded hypersurface in $C_R$ with $\partial V_0\subset \partial C_R$, such that the portion of $V_0$ inside $(C_R\cap S_{[-h,h]})$ is minimal in free boundary sense, with respect to the boundary portion of $\partial C_R$ between the coordinate hyperplanes $S_{\pm h}$.
        \end{lemma}
 \begin{proof}
			By the asymptotic flatness of $(M,g(s))$, we can choose $R_0>1$, such that for $R>R_0$, for any two points $p,q\in\mathbb{R}^n\backslash C_{R-1}$, the following inequality holds
            \begin{align}\label{eq: 45}
            \frac{1}{2}d_{\mathbb{R}^n}(p,q)<d_{g(s)}(p,q)<2d_{\mathbb{R}^n}(p,q)
            \end{align}
   Furthermore, it is clear that $k,i_0$ and $\Lambda$ exists as positive number, owing to the asymptotic flatness. By Lemma \ref{lem: monotonicity} and the monotonicity formula for minimal surface without boundary (\cite{CM11} P234), there exists $L_0, \Lambda_0 >0$ depending only on $k,i_0,\Lambda$, such that for $p\in V\cap (C_R\cap S_{[-h+L_0,h-L_0]})$, the monotonicity formula
			\begin{align*}
				e^{\Lambda_0\sigma}\frac{\mathcal{H}^{n-1}(V_0\cap B_{\sigma}(p))}{\sigma^{n-1}}\le e^{\Lambda_0\rho}\frac{\mathcal{H}^{n-1}(V_0\cap B_{\rho}(p))}{\rho^{n-1}}
			\end{align*}
			holds, provieded one of the following holds:
			
			(1) $p\in \partial M$, $0<\sigma<\rho<L_0$.
			
			(2) $p\in \Int M$, $B_{\tau}(p)\cap \partial V=\emptyset$, $0<\sigma<\rho<\tau<L_0$. 
   
   Here we require $p\in C_R\cap S_{[-h+L_0,h-L_0]}$ to ensure that $B_{\sigma}(p)$ and $B_{\rho}(p)$ lie between $S_{\pm h}$, so the intersection of $B_{\sigma}(p)$ and $B_{\rho}(p)$ with $V_0$ lies in the minimal part of $V_0$.

$\quad$

 \begin{figure}
    \centering
    \includegraphics[width = 11cm]{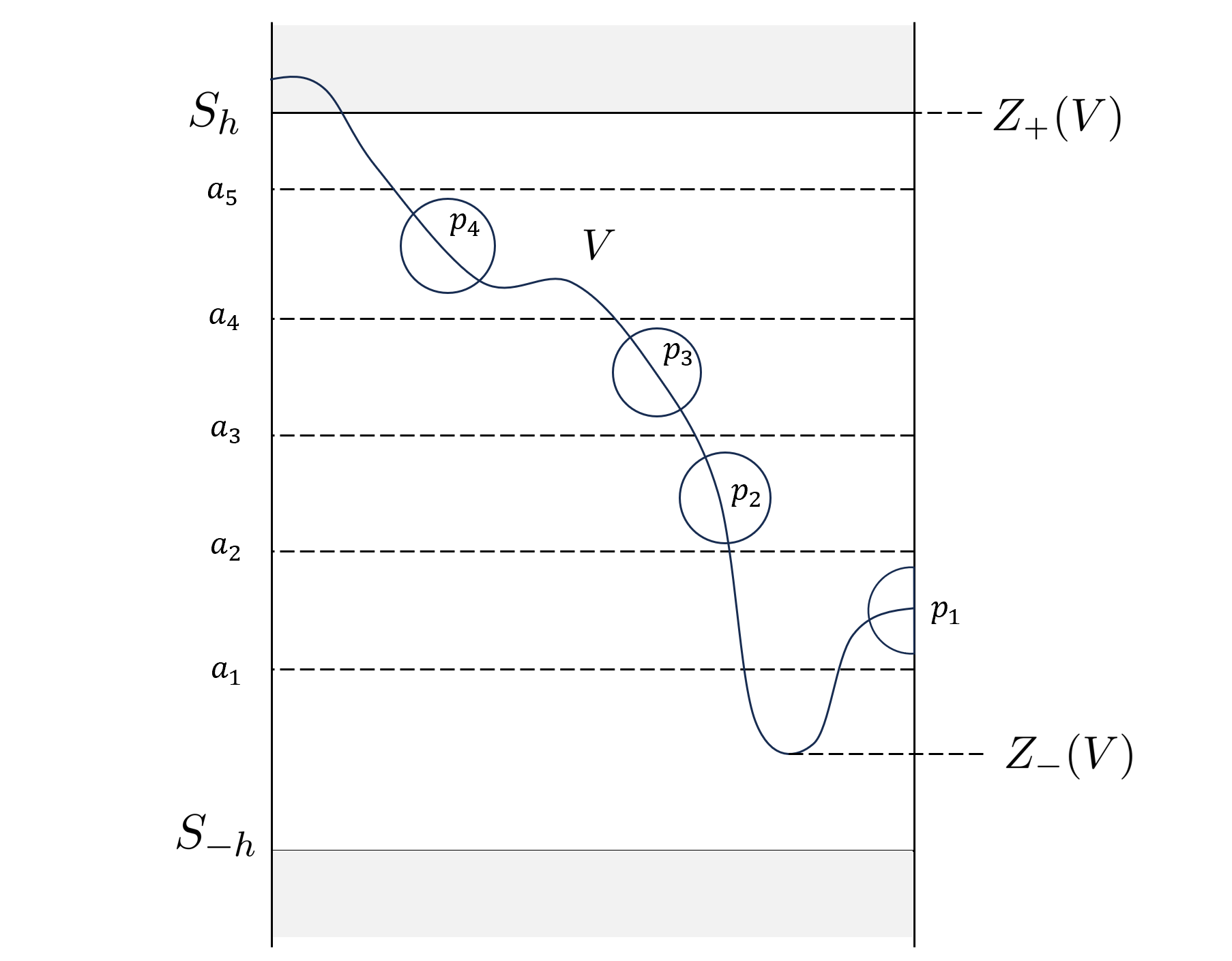}
    \caption{Covering $V$ by small geodesic balls}
    \label{f6}
\end{figure}  

In the following, for simplicity of the notation, we set $V = V_0\cap (C_R\cap S_{[-h,h]})$ to be the minimal portion on $V_0$. Fix $d<\min\lbrace\frac{L_0}{2},\frac{1}{2}\rbrace$. Let us assume that either $|Z_+(V)|$ or $|Z_-(V)|$ is greater than $1$, or else \eqref{eq: 52} automatically holds. Without loss of generality we assume $Z_+(V)>1$.

\textbf{Case 1:} $Z_-(V)<-1$. 

In this case, we can pick $Z_-(V)<a_1<a_2<\dots<a_k = -1$, $1=a_{k+1}<a_{k+2}<\dots<a_l<Z_+(V)$, such that $a_{i+1}-a_i>5d$. Then for sufficiently large $Z(V)$, we can always find $a_i$'s such that
\begin{align*}
    l > &\lceil\frac{Z_+(V)-1}{5d}\rceil+\lceil\frac{-1-Z_-(V)}{5d}\rceil\\
    \ge&\frac{Z(V)-2}{5d}
\end{align*}
Here $\lceil x\rceil$ denotes the minimal integer greater than or equal to $x$. 

\textbf{Case 2:} $Z_-(V)\ge -1$.

In this case, we can pick $\max\{Z_-(V),1\}<a_1<a_2<\dots<a_l <Z_+(V)$, such that $a_{i+1}-a_i>5d$, and
\begin{align*}
    l>\lceil\frac{Z_+(V)-\max\{Z_-(V),1\}}{5d}\rceil>\frac{Z(V)-2}{5d}
\end{align*}
where in the second inequality we have used $\max\{Z_-(V),1\}\le Z_-(V)+2$.

$\quad$

Denote 
   \begin{align*}
       b_i = a_i+d, c_i = a_i+2d, d_i = a_i+3d, e_i = a_i+4d
   \end{align*}
   We aim to pick a point $p_i$ in $S_{[b_i,e_i]}$ for each $i$, and obtain a uniform lower bound estimate of $\mathcal{H}^{n-1}(B_d(p_i)\cap V)$.
			
			If $\Omega_i\cap\partial V\ne \emptyset$, then we simply pick $p_i\in \partial V\cap\Omega_i$. According to the choice of $R_0$ in \eqref{eq: 45}, the geodesic ball $B_{\frac{d}{2}}(p_i)$ always lies in $S_{[a_i,a_{i+1}]}$. By \eqref{eq: monotonicity} we have
			\begin{align}\label{eq: 5}
				\mathcal{H}^{n-1}(B_d(p_i)\cap V)\ge \Theta^{n-1}(V,p_i)\omega_{n-1}(\frac{d}{2})^{n-1} = \frac{1}{2^n}\omega_{n-1}d^{n-1}.
			\end{align}
			
			If $\Omega_i\cap\partial V= \emptyset$, we discuss the following two cases:
			
			\textbf{Case 1:} There exists a point $p_i$ with $\dist(p_i,\partial V)>d$. Then $\partial V\cap B_d(p_i)=\emptyset$. By the monotonicity formula we have
			\begin{align}\label{eq: 6}
				\mathcal{H}^{n-1}(B_d(p_i)\cap V)\ge \Theta^{n-1}(V,p_i)\omega_{n-1}d^{n-1} = \omega d^{n-1}.
			\end{align}
			
			\textbf{Case 2:} $V\cap S_{[b_i,e_i]}$ lies entirely in a $d$-neighbourhood of $\partial C_R$. We pick a point $p_i\in S_{[c_i,d_i]}$. By the choice of $R_0$ in \eqref{eq: 45} we know the geodesic ball $B_{\frac{d}{2}}(p_i)$ lies in $S_{[b_i,e_i]}$. Furthermore, since there is no boundary point of $V$ in  $S_{[b_i,e_i]}$, we have $B_{\frac{d}{2}}(p_i)\cap \partial V = \emptyset$. Consequently, the monotonicity formula applies:
			\begin{align}\label{eq: 7}
				\mathcal{H}^{n-1}(B_{\frac{1}{2}d}(p_i)\cap V)\ge \Theta^{n-1}(V,p_i)\omega_{n-1}(\frac{d}{2})^{n-1} = \frac{\omega_{n-1}}{2^{n-1}} d^{n-1}.
			\end{align}
			
			Since $B_{\frac{d}{2}}(p_i)\subset S_{a_i,a_{i+1}}$, it is straightforward to see the ball we construct are disjoint with each other. Combining \eqref{eq: 5}\eqref{eq: 6}\eqref{eq: 7} we conclude that
			\begin{align*}
				\mathcal{H}^{n-1}(V)>\frac{\omega_{n-1}}{2^n} d^{n-1}l>\frac{\omega_{n-1}}{10 \cdot 2^n} d^{n-2}(Z(V)-2).
			\end{align*}
		\end{proof}

        Let $C_{R_i}^+$ be the region in $C_{R_i}$ beyond $\Sigma_i$. Define
        \begin{align*}
            \mathcal{G}_{R_i} = &\lbrace \Sigma = \partial\Omega\cap\mathring{C}_{R_i}^+, \Omega \mbox{ is a Caccioppoli set in } C_{R_i}^+,\\
      &\Sigma_i\subset\Omega,\quad \Omega\cap\{z\ge h\} = \emptyset \mbox{ for some }h\rbrace
        \end{align*}

	\begin{lemma}\label{lem: Existence}
		For the cylinder $(C_{R_i},g(s))$, there exists an free boundary minimal hypersurface $\Sigma_i(s)$ which minimizes the volume in $\mathcal{G}_{R_i}$. Moreover, $\Sigma_i(s)$ intersects $W$.
		
	\end{lemma}
	
	\begin{proof}
		Without particularly indicated, the Hausdorff measure throughout this proof is considered with respect to metric $g(s)$. For $h$ sufficiently large, denote $C_{R_i,h}$ to be the region in $C_{R_i}$ between the coordinate hyperplane $S_h$ and $\Sigma_i$. By a $C^0$-small $\epsilon$-perturbation of the metric in a $\delta_0$-neighbourhood of $S_{\pm h}$, we obtain a metric $g^{\epsilon}(s)$ on $C_{R_i,h}$, such that the mean curvature on $S_h$ with respect to outer normal is positive, and the dihedral angle of $S_h$ with $\partial C_{R_i}$ is acute inside $C_{R_i,h}$. Here a $C^0$-small $\epsilon$-perturbation means that
		\begin{align*}
			(1-\epsilon)g^{\epsilon}(t)(v,v)<g(s)(v,v)<(1+\epsilon)g^{\epsilon}(t)(v,v)
		\end{align*}
for all $v\in TC_{R_i,h}$. $C_{R_i,h}$ then satisfies the condition in Lemma \ref{lem: free boundary existence} satisfied by $N$, with $\partial_+N = C_{R_i}\cap S_h$, $\partial_-N = \Sigma_i$. Define
\begin{align*}
		    \mathcal{G}_{R_i,h} = &\lbrace \Sigma = \partial\Omega\cap\mathring{C}_{R_i,h}, \Omega \mbox{ is a Caccioppoli set in } C_{R_i,h},\\
      &\Sigma_i\subset\Omega,\quad (C_{R_i}\cap S_h)\cap\Omega = \emptyset \rbrace
		\end{align*}
  By Lemma \ref{lem: free boundary existence}, there exists a volume minimizing free boundary minimal hypersurface $\Sigma_i^h(s)$ in the class $\mathcal{G}_{R_i,h}$.

 \begin{figure}
    \centering
    \includegraphics[width = 15cm]{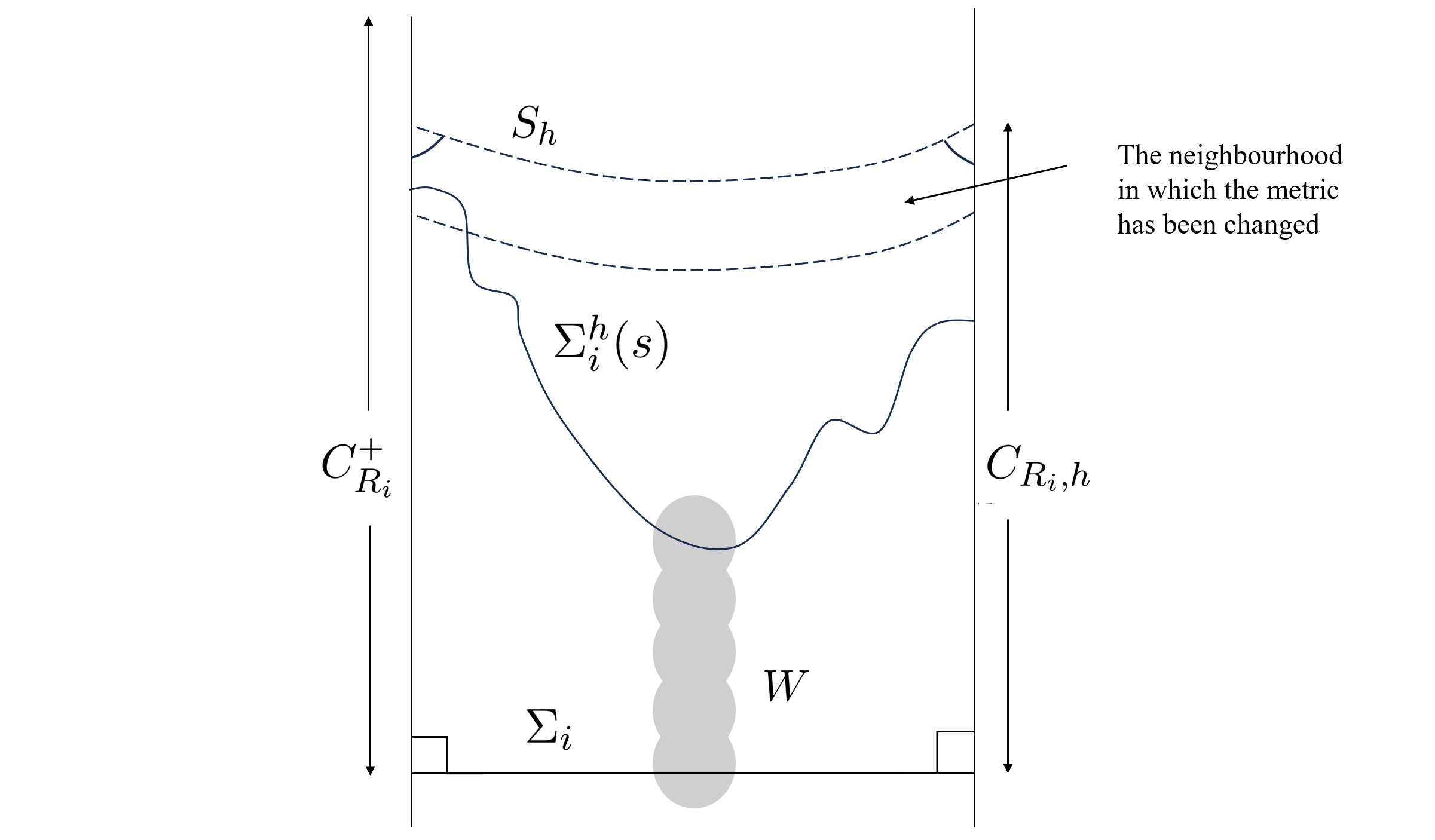}
    \caption{A diagram for $C_{R_i,h}$}
    \label{f5}
\end{figure}  

		We claim for $\epsilon$ sufficiently small, $\Sigma_i^h(s)$ must intersect with $W$. Or else, it follows that
		\begin{align*}
			&\mathcal{H}^{n-1}_g(\Sigma_i^h(s)) = \mathcal{H}^{n-1}_{g(s)}(\Sigma_i^h(s))\\
			&<(1+\epsilon)\mathcal{H}^{n-1}_{g^{\epsilon}(s)}(\Sigma_i^h(s))\le (1+\epsilon)\mathcal{H}^{n-1}_{g^{\epsilon}(s)}(\Sigma_i)\\
			&<\frac{1+\epsilon}{1-\epsilon}\mathcal{H}^{n-1}_{g(s)}(\Sigma_i) = \mathcal{H}^{n-1}_{g}(\Sigma_i)-(\mathcal{H}^{n-1}_{g}(\Sigma_i)-\mathcal{H}^{n-1}_{g(s)}(\Sigma_i))+\frac{2\epsilon}{1-\epsilon}\mathcal{H}^{n-1}_{g(s)}(\Sigma_i)\\
			&\le \mathcal{H}^{n-1}_{g}(\Sigma_i^h(s))-(\mathcal{H}^{n-1}_{g}(\Sigma_i)-\mathcal{H}^{n-1}_{g(s)}(\Sigma_i))+\frac{2\epsilon}{1-\epsilon}\mathcal{H}^{n-1}_{g(s)}(\Sigma_i)
		\end{align*}
		In the second inequality in line 2 we have used the volume minimizing property of $\Sigma_i^h(s)$ under the metric $g^{\epsilon}(t)$ and for the last inequality we have used the volume minimizing property of $\Sigma_i$ under the metric $g$. This implies
		\begin{align*}
			0<\mathcal{H}^{n-1}_{g}(\Sigma_i)-\mathcal{H}^{n-1}_{g(s)}(\Sigma_i)<\frac{2\epsilon}{1-\epsilon}\mathcal{H}^{n-1}_{g(s)}(\Sigma_i)
		\end{align*}
		A contradiction then follows by letting $\epsilon\to 0$.
		
		Next, we show for $h$ sufficiently large, $\Sigma_i^h(s)$ does not touch the small neighbourhood of $S_h$ where the metric has been changed to $g^{\epsilon}(s)$. We argue by contradiction. Pick a component $\hat{\Sigma}$ of $\Sigma_i^h(s)$ which intersects the $\delta_0$ neighbourhood of $S_h$. We show that $\hat{\Sigma}$ must intersect $K\cup W$. If not, then $\hat{\Sigma}$ lies entirely in $\mathbb{R}^n\backslash B_1^n(O)$. Denote $\Gamma = \partial\hat{\Sigma}$, and assume that $\Sigma_i^h(s) = \partial\Omega\cap\mathring{C_{R_i,h}}$. We discuss following cases respectively

  (1) $\Gamma = \emptyset$, then $\hat{\Sigma}$ is a closed hypersurface in $\mathbb{R}^n\backslash B^n_1(O)$. By Jordan-Brouwer separation Theorem, we find an open set $U\subset M$ with compact closure satisfying $\partial U = \hat{\Sigma}$. Replacing $\Omega$ by $\Omega\cup U$ eliminates the component $\hat{\Sigma}$, and the volume of $\Sigma_i^h(s)$ decreases. A contradiction.

  (2) $\Gamma$ bounds a compact region $P$ on $\partial C_{R_i}\cong S^{n-2}\times\mathbb{R}$. In this case, $P\cup \hat{\Sigma}$ is a closed hypersurface, and the argument for (1) carries through.

  (3) $\Gamma$ does not bound compact region in $\partial C_{R_i}$. In this case, by applying Jordan-Brouwer Theorem to two point compactification of $S^{n-2}\times\mathbb{R}$ diffeomorphic to $S^{n-1}$, we see $\Gamma$ separates $\partial C_{R_i}$ into two noncompact regions, $P_-$ and $P_+$. Again by Jordan-Brouwer Theorem,  $\hat{\Sigma}\cup P_-$ bounds a noncompact region $\Omega'$, so we have $\hat{\Sigma}\in\mathcal{G}_{R_i,h}$, and it must intersect with $W$.

  Therefore, we have found a component $\hat{\Sigma}$ of $\Sigma_i^h(s)$, intersecting both $S_{h-\delta_0}$ and $K\cup W$. The portion of $\Sigma_i^h(s)$ inside $C_{R_i,h-\delta_0}$ is minimal in free boundary sense. By Lemma \ref{lem: height estimate} we deduce
        \begin{align}\label{eq: 46}
            \mathcal{H}^{n-1}_{g(s)}(\Sigma_i^h(s)\cap C_{R_i,h-\delta_0})\ge \eta Z(\Sigma_i^h(s)\cap C_{R_i,h-\delta_0})
        \end{align}
        Note that
        \begin{equation}\label{eq: 47}
            \begin{split}
                &Z_+(\Sigma_i^h(s)\cap C_{R,h-\delta_0}) = h-\delta_0\\
            &Z_-(\Sigma_i^h(s)\cap C_{R,h-\delta_0}) \le \max\lbrace1,Z_+(K\cup W)\rbrace
            \end{split}
        \end{equation}
         On the other hand, we note that
        \begin{equation}\label{eq: 64}
            \begin{split}
                \mathcal{H}^{n-1}_{g(s)}(\Sigma_i^h(s)\cap C_{R,h-\delta_0})\le\mathcal{H}^{n-1}_{g(s)}(\Sigma_i)<\mathcal{H}^{n-1}_{g}(\Sigma_i)\\
            \le\liminf_{t\to\infty}\mathcal{H}^{n-1}_g(C_{R_i}\cap S_t) = \omega_{n-1}R_i^{n-1}
            \end{split}
        \end{equation}
        By Letting $h\to\infty$, \eqref{eq: 64} contradicts with \eqref{eq: 46}\eqref{eq: 47}.

		Consequently, there exists $h_0>0$, for $h>h_0$, $\Sigma_i^h(s)$ does not touch the region where $g(s)$ has been changed. This shows $\Sigma_i^h(s)$ is entirely minimal under the metric $g(s)$. Therefore
            \begin{align}\label{eq: 65}
            \mathcal{H}^{n-1}(\Sigma_i^h(s))=\mathcal{H}^{n-1}(\Sigma_i^{h_0}(s)) \mbox{ for all }h>h_0
            \end{align}
  since $\Sigma_i^h(s)$ lies entirely in $C_{R,h_0}$.
		
		Finally, it remains to show $\Sigma_i^h(s)$ minimizes the volume globally in $\mathcal{G}_{R_i}$ for $h>h_0$. If not, there exists $\Sigma'\subset\mathcal{G}_{R_i}$ with $\mathcal{H}^{n-1}(\Sigma')<\mathcal{H}^{n-1}(\Sigma_i^h(s))=\mathcal{H}^{n-1}(\Sigma_i^{h_0}(s))$. Then by taking $h$ sufficiently large such that $\Sigma'\subset C_{R_i,h}$, we obtain a contradiction with the minimizing property of $\Sigma_i^h(s)$! Set ${\Sigma}_i(s)=\Sigma_i^h(s)$ for any $h>h_0$. Though the choice of ${\Sigma}_i(s)$ may not be unique, each choice has the same volume due to \eqref{eq: 65}, and it turns out ${\Sigma}_i(s)$ minimizes the volume in $\mathcal{G}_{R_i}$. This completes the proof of Lemma \ref{lem: Existence}.
		
	\end{proof}

	By letting $R_i\to\infty$, We obtain a sequence of free boundary volume minimizing surface $\Sigma_i(s)$ under the metric $g(s)$, each of which intersects with the compact set $W$. Since $\Sigma_i(s)$ is a minimizing boundary it is standard to obtain the following volume estimate
	\begin{align}\label{eq: 4}
		\mathcal{H}^{n-1}(\Sigma_i(s)\cap B_R)\le \mathcal{H}^{n-1}(\partial B_R)\le \omega_{n-1}R^{n-1}.
	\end{align}
	Here $B_R$ is the coordinate ball in the asymptotically flat end. It then follows from standard results in geometric measure theory that $\Sigma_i(s)$ converges subsequencially to  a complete and noncompact properly embedded hypersurface $\Sigma(s)$ in $(M,g(s))$, satisfying the same volume estimate as \eqref{eq: 4}.

	\begin{lemma}\label{prop: AF surface}
		There exists $r_1>0$, such that $\Sigma(s)$ is a graph over $S_0\setminus B^n_{r_1}$, and  for any $\epsilon>0$ the graph function $u$ satisfies
        \begin{align}\label{eq: 48}
			|y|_{\bar{g}}^{-1}|u(y)-a|+|\Bar{\nabla}u(y)|_{\Bar{g}}+|y|_{\bar{g}}|\Bar{\nabla}^2u(y)|=O(|y|^{-\tau+\epsilon})
		\end{align}
        Consequently, $\Sigma(s)$ is an asymptotically flat manifold with mass zero.
	\end{lemma}

	\begin{proof}		
		Arguing as in Lemma \ref{lem: Sigma parallel to S0}, we see $\Sigma(s)$ has exactly one end. By Proposition 9 in \cite{EK23}, $\Sigma(s)$ is a graph with desired decay estimate as \eqref{eq: 48} over a hyperplane $S'$. If $S'$ is not parallel to $S_0$, then $\Sigma(s)$ would pass through $\Sigma$ by Lemma \ref{lem: Sigma parallel to S0}. This contradicts with the fact that $\Sigma(s)$ lies in $M^+$. The fact that $\Sigma(s)$ is an asymptotically flat manifold with mass zero follows immediately from \eqref{eq: 48} and the assumption that $\tau>n-3$
	\end{proof}
	
	\begin{lemma}\label{lem: ACV}
		$\Sigma(s)$ is stable under asymptotically constant variation.
	\end{lemma}
	
	\begin{proof}
		Let $X$ be a vector field on $\Sigma(s)$ which equals $\frac{\partial}{\partial t}$ outside a compact set and let $F_t: M\longrightarrow M$ be the one parameter transformation group of $X$. For an embedded hypersurface $V$ and a point $p\in V$, denote $a(p)$ to be the second derivative of the volume element of $F_t(V)$ with respect to $t$, evaluated at $t = 0$. It follows from \cite{Simon83} that
        \begin{align*}
            a = &-\sum_{i=1}^{n-1}R(X,e_i,X,e_i)+\divv_VZ+(\divv_VX)^2\\
            &+\sum_{i=1}^{n-1}|(D_{e_i}X)^{\perp}|^2-
            \sum_{i,j=1}^{n-1}\langle e_i,D_{e_j}X\rangle\langle e_j,D_{e_i}X\rangle
        \end{align*}
        Here, $\{e_i\}_{1\leq i\leq n-1}$ forms an locally orthonormal basis on $V$. $Z = D_{\frac{\partial}{\partial t}}X$ is used to denote the acceleration vector field of the transformation $F_t$, and $D$ denotes the covariant differentiation in $M$ with respect to $g(s)$. Since $\Sigma_i(s)$ has least volume in the free boundary class in $C_{R_i}$, we have
        \begin{align}\label{eq: 49}
            \int_{\Sigma_i(s)}ad\mu\ge 0
        \end{align}
        for $R_i$ sufficiently large.

        By the asymptotic flatness, we have the following hold along each $\Sigma_i(s)$ and $\Sigma(s)$ outside a compact set:
        \begin{equation}\label{eq: 53}
            \begin{split}
                &Z = D_{\frac{\partial}{\partial t}}\frac{\partial}{\partial t} = O(|x|^{-\tau-1})\\
                &\divv_{\Sigma_i(s)}Z = O(|x|^{-\tau-2}), \divv_{\Sigma(s)}Z = O(|x|^{-\tau-2})\\
            &D_{e_i}X = D_{e_i}\frac{\partial}{\partial t} = O(|x|^{-\tau-1})\\
            &R(X,e_i,X,e_i) = O(|x|^{-\tau-2})
            \end{split}
        \end{equation}
        It follows that $a = O(|x|^{-\tau-2})$ along each $\Sigma_i(s)$ and $\Sigma(s)$. As $\tau+2>n-3+2=n-1$, $a$ must integrable over $\Sigma(s)$. Combining with \eqref{eq: 49} and the dominated convergence theorem we conclude that
        \begin{align*}
             \int_{\Sigma(s)}ad\mu\ge 0
        \end{align*}

        We write $X = \hat{X}+\varphi \nu$, $Z = \hat{Z}+\psi\nu$, with $\hat{X},\hat{Z}$ tangent to $\Sigma(s)$ and $\nu$ the unit normal, $\hat X:=\sum^{n-1}_{i=1}\hat X_i e_i$, $\hat X_{i,j}:=<D_{e_i}\hat X, e_j>$. By the calculation of \cite{Schoen1989} we have
        \begin{align*}
            G = &-2\sum_{i=1}^{n-1}\varphi\langle R(\hat{X},e_i)\nu,e_i\rangle-\sum_{i=1}^{n-1}\langle R(\hat{X},e_i)\hat{X},e_i\rangle\\
            &+div\hat{Z}+(div\hat{X})^2-2A(\nabla\varphi,\hat{X})+\sum_{i=1}^{n-1}A(e_i,\hat{X})^2\\
            &-2\varphi\sum_{i,j=1}^{n-1}A(e_i,e_j)\hat{X}_{i;j}-\sum_{i,j=1}^{n-1}\hat{X}_{i;j}\hat{X}_{j;i}
        \end{align*}
        Moreover, it follows from the integration by part in \cite{Schoen1989} that, for a bounded domain $D\subset \Sigma(s)$, there holds
        \begin{align}\label{eq: 54}
            \int_DGd\mu = \int_{\partial D}(div\hat{X})\langle\hat{X},\eta\rangle-\sum_{i,j}\hat{X}_{i;j}\hat{X}_j\eta_i\\
            -2\varphi\sum_{i,j}A(e_i,e_j)\hat{X}_i\eta_j+\langle\hat{Z},\eta\rangle d\sigma
        \end{align}
        Here $\eta$ denotes the outer normal of $\partial D$ on $\Sigma(s)$ and $d\sigma$ denotes the volume element on $\partial D$.

        By \eqref{eq: 48} we have
        \begin{align*}
            |A| = O(|y|^{-\tau- 1+\epsilon})
        \end{align*}
        Moreover, from \eqref{eq: 53} we see that outside a compact set the following holds:
        \begin{align*}
             \hat{X}_{i;j} &= \langle\nabla_{e_i}\hat{X},e_j\rangle = \langle \nabla_{e_i}X-\varphi\nu,e_j\rangle\\
             &=\langle \nabla_{e_i}\frac{\partial}{\partial t},e_j\rangle - \varphi A(e_i,e_j)\\
             &= O(|y|^{-\tau-1})\\
            \langle \hat{Z},\eta\rangle &= \langle Z,\eta\rangle = O(|y|^{-\tau-1+\epsilon})
        \end{align*}
        Here, $y$ is as in Lemma \ref{prop: AF surface} to describe $\Sigma(s)$ as a graph outside a compact set. By choosing $\epsilon$ sufficiently small we have
        \begin{align*}
            G = O(|y|^{2-n-\epsilon})
        \end{align*}
        Let $D = B_d$ be the $d$-geodesic ball on $\Sigma(s)$ and let $d\to\infty$, we obtain
        \begin{align}\label{eq: 61}
            \int_{\Sigma(s)}|\nabla\varphi|^2d\mu\ge\int_{\Sigma(s)}(Ric(\nu,\nu)+|A|^2)\varphi^2d\mu
        \end{align}
        which holds for all the test function $\varphi$ coinciding with $\varphi_0 = \langle \frac{\partial}{\partial t},\nu\rangle$ outside a compact set.

        From $\nu = \Bar{\nu}+O(|y|^{-\tau})$ we obtain
        \begin{align*}
            \varphi_0 = \langle \frac{\partial}{\partial t},\Bar{\nu}\rangle = \frac{1}{(1+|\Bar{\nabla}u|^2)^{\frac{1}{2}}}+O(|y|^{-\tau}) = 1+O(|y|^{-\tau+\epsilon})
        \end{align*}
        It follows that $\varphi_0-1$ is $L^2$-integrable out side a compact set. Let $\phi$ be a $C^1$ function which equals to $1$ outside a compact set and $\varphi$ as above. Let $\zeta_d$ be a cut off function which equals $1$ inside the geodesic ball $B_d$ in $\Sigma(s)$ and $0$ outside $B_{2d}$. Define
        \begin{align*}
            \phi_d = \zeta_d\phi+(1-\zeta_d)\varphi
        \end{align*}
       Note that $Ric(v,v)+|A|^2=O(|y|^{-\tau-2})$,  and $\phi_d\to\phi$ in $W^{1,2}$ as $d\to\infty$. It follows that \eqref{eq: 61} also holds for $\phi$. This completes the proof of Lemma \ref{lem: ACV}.
 \end{proof}

 \begin{remark}\label{rk: finite energy}
     By the same argument, we see that the conclusion remains true, provided that $\phi-1$ has finite $W^{1,2}$ norm and $\phi$ is asymptotic to a constant at infinity. See also Remark 29 in \cite{EK23}.
 \end{remark}

The following infinitesimal rigidity lemma is a consequence of Lemma \ref{lem: ACV}, which follows from the argument in \cite{Carlotto16} and \cite{EK23}, and originates from the proof of the positive mass theorem \cite{SY79} and \cite{Schoen1989}.
\begin{lemma}\label{lem: infinitesimal rigidity}
    Suppose the ambient scalar curvature is nonnegative along $\Sigma(s)$, then $\Sigma(s)$ is isometric to $\mathbb{R}^n$ and totally geodesic. Moreover, we have $R = Ric(\nu,\nu) = 0$ along $\Sigma(s)$.
\end{lemma}

\begin{proof}

    Since the argument is well known to the experts we only give a sketch of the proof, and for a detailed treatment the readers may refer to P.14, P.15 in \cite{Carlotto16}. Let $\phi$ be the solution of the conformal Laplacian equation on $\Sigma(s)$ which equals to $1$ at infinity. It follows that $(\Sigma(s),\phi^{\frac{4}{n-3}}g(s))|_{\Sigma(s)}$ is an asymptotically flat manifold with zero scalar curvature. The stability of $\Sigma(s)$ under asymptotically constant variation implies that the mass of $\Sigma(s)$ is non-increasing under the conformal deformation. Thus, the mass of $\Sigma(s)$ is nonpositive due to Lemma \ref{prop: AF surface}. By the positive mass theorem we see $(\Sigma(s),\phi^{\frac{4}{n-3}}g(s))|_{\Sigma(s)}$ is isometric to $\mathbb{R}^{n-1}$. Again by Lemma \ref{lem: ACV}, one deduces $\phi$ is a constant, so $\Sigma(s)$ is isometric to $\mathbb{R}^{n-1}$. By letting $\phi = 1$ in \eqref{eq: 48} and using the rearrange formula
    \begin{align*}
        Ric(\nu,\nu)+|A|^2 = \frac{1}{2}(|A|^2+R)
    \end{align*}
    we see $\Sigma(s)$ is totally geodesic, which $R=Ric(\nu,\nu)$ along $\Sigma(s)$.
\end{proof}

\begin{lemma}\label{lem: coincide}
    There exists $t_1>t_0$ with the following property: If $p$ is chosen in $\Sigma_t$ with $t\ge t_1$, then by letting $r\to 0$, $s\to 0$, $\Sigma(s)$ subconverges locally smoothly to an area minimizing hypersurface $\Sigma_p$ passing through $p$. Moreover, $\Sigma_p = \Sigma_t$.
\end{lemma}
\begin{proof}
    From Lemma \ref{lem: Existence} we know $\Sigma_i(s)\cap W\ne\emptyset$ for each $i$. By letting $i\to\infty$, we have $\Sigma(s)\cap\bar{W}\ne\emptyset$. We discuss the following two situations:

    \textbf{Case 1:} $\Sigma(s)\cap W\ne\emptyset$. 

    In this case, $\Sigma(s)$ must intersect $B_r(p)$. Or else, the ambient scalar curvature turns out to be nonnegative along $\Sigma(s)$, and strictly positive somewhere on $\Sigma(s)$, due to Lemma \ref{lem: g(s) perturbation}, which violates Lemma \ref{lem: infinitesimal rigidity}. Consequently, by letting $r\to 0$, $s\to 0$, $\Sigma(s)$ must subconverge locally smoothly to an area minimizing hypersurface $\Sigma_p$ passing through $p$, with respect to metric $g$. 
    
    To see $\Sigma_p$ coincides with $\Sigma_t$, recall that each $\Sigma(s)$ is asymptotic to certain hyperplane parallel to $S_0$ from Lemma \ref{prop: AF surface}. It follows that $\Sigma_p$ is also asymptotic to a hypersurface $S_{t'}$ for some $t'$. If $t'>t_0$, by uniqueness statement Proposition \ref{prop: uniqueness} we see that $\Sigma_p$ coincides with $\Sigma_{t'}$. As $p\in\Sigma_t$, we obtain $t'=t$. If $t'<t_0$, By increasing $\zeta$ we see $\Sigma_p$ must touch $\Sigma_{\zeta}$ for some $\zeta>t_0$, so $\Sigma_p$ coincides with $\Sigma_{\zeta}$. Given the fact that $p\in\Sigma_t$ we see $\Sigma_{\zeta}$ = $\Sigma_t$ and the conclusion follows.

    \textbf{Case 2:} $\Sigma(s)\cap W=\emptyset$. 

    In this case, $\Sigma(s)$ must touch $\partial W$ somewhere. For any point $p\in M\backslash K$ we introduce following notations.
    Denote
    \begin{align*}
        &l_p^+ = \{q\in M_+, x_j(q) = x_j(p), j=1,2,\dots,n-1; z(q)\ge z(p)\}\\
        &l_P^- = \{q\in M_+, x_j(q) = x_j(p), j=1,2,\dots,n-1; z(q)\le z(p)\}
    \end{align*}

     \begin{figure}
    \centering
    \includegraphics[width = 15cm]{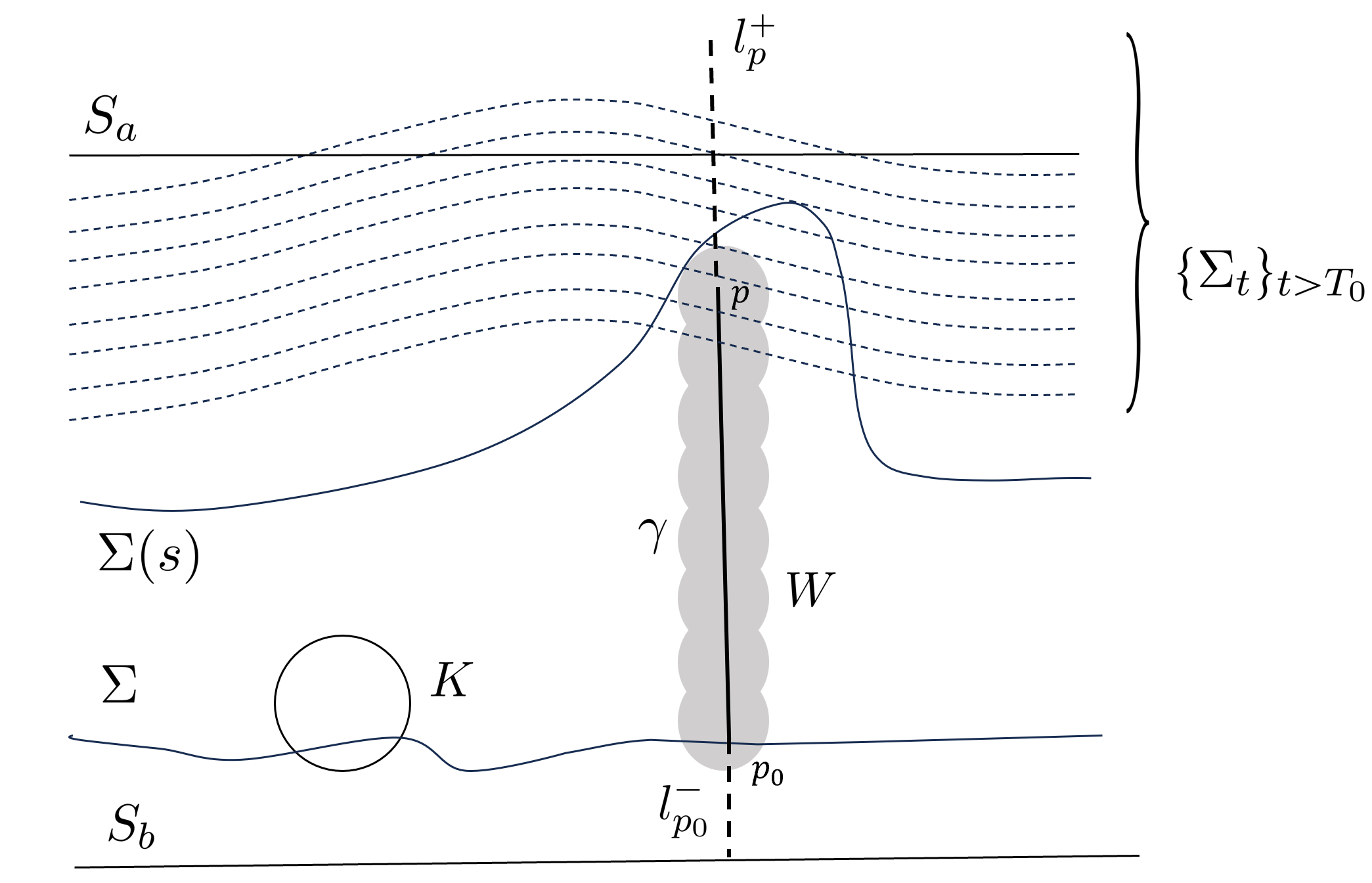}
    \caption{The case that $\Sigma(s)$ touches $\partial W$ in $M$}
    \label{f2}
\end{figure}   
    We hope to verify the following topological claim

    \textbf{Claim} $\Sigma(s)$ intersects $l_p^+$.

    \begin{proof}[Proof of the Claim]
        Assume $\Sigma(s)$ is asymptotic to $S_{s'}$. We pick $a$ sufficiently large such that both $W$ and $\Sigma(s)$ lie below $S_a$. Moreover, we pick $b<-1$ such that both $\Sigma$ and $K$ lie beyond $S_b$. Let $\Omega$ be the region between $S_a$ and $S_b$, so $\Omega$ is homeomorphic to $\Omega_0\cup_{S^{n-1}}K$, with $\Omega_0 = \mathbb{R}^{n-1}\times[a,b]\backslash B_1^n(0)$. 

        Let
        \begin{align*}
            l = (l_p^+\cup \gamma\cup l_{p_0}^-)\cap \Bar{\Omega}
        \end{align*}
        be a curve joining $S_a$ and $S_b$. From the long exact sequence of relative homology group and the fact $\mathbb{R}^{n-1}\times[a,b]\backslash l\cong (\mathbb{R}^{n-1})\backslash \{0\}\times[a,b]$ we know that
        \begin{align}\label{eq: 56}
            H_{n-1}(\Omega_0\backslash l) = \mathbb{Z}\oplus\mathbb{Z}
        \end{align}
        By Lemma \ref{prop: AF surface} $\Sigma(s)$ is asymptotic to $S_{s'}$, so $\Sigma(s)$ is a graph over $S_{s'}$ outside a compact set. Consider a disk $D_{\rho}\subset S_{s'}$. By choosing $\rho$ sufficiently large, $\partial D_{\rho}$ represents the generator $\alpha$ of the first $\mathbb{Z}$ summand in \eqref{eq: 56}. The generator $\beta$ of the second $\mathbb{Z}$ summand is represented by the inner boundary of $\Omega_0$: $\partial B_1^n(0)\cong S^{n-1}$. By the Mayer-Vietoris sequence for $\Omega_0$ and $K$ we see that
        \begin{align*}
        0\longrightarrow H_{n-1}(S^{n-1}) \stackrel{i_1\oplus i_2}{\longrightarrow} H_{n-1}(\Omega_0)\oplus H_{n-1}(K)\stackrel{j}{\longrightarrow} H_{n-1}(\Omega)\longrightarrow 0
        \end{align*}
        Since $i_1$ maps the generator of $H_{n-1}(S^{n-1})=\mathbb{Z}$ to $\beta$, the first summand in \eqref{eq: 56} is mapped injectively by $j$ into $H_{n-1}(\Omega)$. This shows $\partial D_{\rho}$ is not homologuous to zero in $H_{n-1}(\Omega)$, and the same thing holds for $\{(x,u_s(x)), x\in \partial D_{\rho}\}$, where $u_s$ denotes the graph function of $\Sigma(s)$ outside a compact set due to Lemma \ref{prop: AF surface}.

        On the other hand, by Lemma \ref{prop: AF surface}, we know $\{(x,u_s(x)), x\in \partial D_{\rho}\}$ is the boundary of a compact region in $\Sigma(s)$. This shows $\Sigma(s)$ intersects with $l$. However, $\Sigma(s)\cap W = \emptyset$, and $\Sigma(s)$ lies beyond $\Sigma$. This shows $\Sigma(s)\cap l_p^+\ne\emptyset$.
        \end{proof}

        Since $\Sigma(s)\cap W = \emptyset$, $\Sigma(s)$ turns out to be an area minimizing hypersurface with respect to the original metric $g$. By the same argument as in Case 1, we see that $\Sigma(s)$ coincide with $\Sigma_{t'}$ for $t'>t>t_1$. Let $\epsilon_0$ be as in Lemma \ref{lem: g(s) perturbation}. When $t_1$ is sufficiently large, the slope of $\Sigma_{t'}$ is smaller than $\epsilon_0$. Assume $\Sigma_{t'}$ touches $\partial W$ at $q$, then the slope of $q$ with respect to $\partial W$ is also smaller than $\epsilon_0$. It follows from Lemma \ref{lem: g(s) perturbation} that $\dist(q,p)<6r$. By Letting $r\to 0$, we see that $\Sigma(s)$ subconverge to an area minimizing hypersurface $\Sigma_p$ passing through $p$, which in turn coincides with $\Sigma_t$.

\end{proof}
	
	\begin{lemma}\label{lem: limit strong stability}
		For $t\ge t_1$, $\Sigma_t$ is stable under asymptotically constant variation.
	\end{lemma}
	\begin{proof}
		Let $\varphi\in C^{\infty}(\Sigma_t)$ which equals to $1$ outside a compact set on $\Sigma_t$. We can extend $\varphi$ to a smooth function $\phi$ defined on $M$, such that $\phi$ equals to $1$ outside a compact set. By the previous Lemma, there exists $s_j\to 0$, such that $\Sigma(s_j)$ converge to $\Sigma_t$ locally smoothly. Therefore, on each ball $B_R\subset M$, there holds
        \begin{align}\label{eq: 59}
            \int_{\Sigma(s_j)\cap B_R}(Ric(\nu,\nu)+|A|^2)\phi^2d\mu\to \int_{\Sigma_t\cap B_R}|(Ric(\nu,\nu)+|A|^2)\phi^2d\mu
        \end{align}
        Moreover, by uniform decay estimate in Lemma \ref{lem: decay estimate for higher derivative of u}, we have
        \begin{align}\label{eq: 60}
            \int_{\Sigma(s_j)\backslash B_R}(Ric(\nu,\nu)+|A|^2)\phi^2d\mu = o(1)
        \end{align}
        where the $o(1)$ is independent of $R$. By letting $R\to\infty$, in conjunction with Lemma \ref{lem: ACV},\eqref{eq: 59}\eqref{eq: 60}, we see $\Sigma_t$ is stable under variation which equals $1$ outside a compact set. Remark \ref{rk: finite energy} implies the same thing holds true when $\varphi$ is only required to be asymptotic to $1$ in $W^{1,2}$ sense.
	\end{proof}

        \begin{proof}[Proof of Theorem \ref{thm: free boundary}]
            By Lemma \ref{prop: AF surface}, Lemma \ref{lem: coincide} and Lemma \ref{lem: limit strong stability} , we see that for $t\ge t_1$, $\Sigma_t$ is an area minimizing hypersurface which is asymptotically flat under induced metric, with zero mass and stable under asymptotic constant variation. It follows from Lemma \ref{lem: infinitesimal rigidity} that $\Sigma_t$ is isometric to $\mathbb{R}^{n-1}$, totally geodesic and has normal Ricci curvature vanishing. By Proposition \ref{flatness}, we deduce that the region beyond $\Sigma_{t_1}$ is isometric to $\mathbb{R}^n_+$. Similarly, there exists $t_2<0$ such that the region below $\Sigma_{t_2}$ is isometric to $\mathbb{R}^n_-$. By the definition of ADM mass we have
            \begin{align}\label{eq: 62}
                m = \lim_{r\to\infty}\frac{1}{2(n-1)\omega_{n-1}}\int_{\partial B_r(O)}(g_{ij,i}-g_{ii,j})\nu^j d\Bar{\mu}
            \end{align}
            with $\partial B_r(O)$ denoting the centered coordinate sphere of radius $r$. Since $\tau>n-3$, the term $(g_{ij,i}-g_{ii,j})$ decays faster than $|x|^{2-n}$. Let us denote the region between $\Sigma_{t_2}$ and $\Sigma_{t_1}$ by $\Omega_{t_2,t_1}$, then it follows that
            \begin{align*}
                \mathcal{H}^{n-1}(\partial B_r(O)\cap\Omega_{t_2,t_1}) = O(r^{n-2})
            \end{align*}
            where the Hausdorff measure is associated with Euclidean metric in $\mathbb{R}^n$. In conjunction with \eqref{eq: 62}, we see $m=0$. This is a contradiction.
        \end{proof}

        \begin{proof}[Proof of Theorem \ref{thm: cylinder minimizing}]
            Let $\Sigma$ be the globally minimizing hypersurface as in Definition \ref{global min}, then for each $R_i$, $\Sigma_i = \Sigma\cap C_{R_i}$ serves as the free boundary minimizing hypersurface in Theorem \ref{thm: free boundary}. Furthermore, these $\Sigma_i$ does not drift to infinity as each of them contains $\Sigma_1$. It follows from Theorem \ref{thm: free boundary} that $m=0$, and $M$ is isometric to $\mathbb{R}^n$.
        \end{proof}

 \appendix
\section{Some $L^p$ estimate for solutions to elliptic equations}

Let $(\mathbf{R}^m, g)$ be a complete Riemannian manifold with $g$ equivalent to the Euclidean metric. Then we have the following Sobolev inequality
\[
(\int_{\mathbf{R}^m} |w|^{\frac{2m}{m-2}}dy)^{\frac{m-2}{m}}\le  C_m\int_{\mathbf{R}^m}|Dw|^2dy.
\]
for any $u\in C^{\infty}_c(\mathbf{R}^m)$.
\begin{lemma}\label{lem: L^ 2m/m-2 estimate}
   Let $w\in C^2(\mathbf{R}^m)\cap W^{1,2}(\mathbf{R}^m)\cap L^{\frac{2m}{m-2}}(\mathbf{R}^m)$ be the solution of the equation
\begin{align}\label{eq: linear eq1}
    a_{ij}D_{ij}w+b_iD_i w +cw = f
\end{align}
on $\mathbf{R}^m$ with uniform elliptic condition
\begin{align*}
    \lambda |\xi|^2\le a_{ij}\xi^i\xi^j
\end{align*}
Then
\begin{equation}\label{eq: L^p bound}
    \begin{split}
        &(\int_{\mathbf{R}^m} |w|^{\frac{2m}{m-2}}dy)^{\frac{m-2}{2m}}\\
    \le & \frac{C_m}{\lambda^2}\left((\int_{\mathbf{R}^m}|(D_ja_{ij}-b_i)^2+2\lambda c|^{\frac{m}{2}} dy)^{\frac{2}{m}}(\int_{\mathbf{R}^m} w^{\frac{2m}{m-2}})^{\frac{m-2}{2m}} + (\int_{\mathbf{R}^m} |f|^{\frac{2m}{m+2}})^{\frac{m+2}{2m}} \right)
    \end{split}
\end{equation}
\end{lemma}

\begin{proof}
\eqref{eq: linear eq1} can be rewritten as
\begin{align}\label{eq: 35}
    D_i(a_{ij}D_jw)+(b_i-D_ja_{ij})D_iw+cw = f
\end{align}
Multiply $w$ on two sides on \eqref{eq: 35} and integration by parts, we get
\begin{align}
    \int_{\mathbf{R}^m} a_{ij}D_iwD_jw dy= \int_{\mathbf{R}^m}(b_i-D_ja_{ij})wD_iw + cw^2 -fw dy
\end{align}

We have
\begin{align*}
    &\lambda \int_{\mathbf{R}^m}|Dw|^2dy \le  \int_{\mathbf{R}^m}(b_i-D_ja_{ij})wD_iw + cw^2 -fw dy\\
    \le & \frac{\lambda}{2}\int_{\mathbf{R}^m} |D_iw|^2 dy + \frac{1}{2\lambda}\int_{\mathbf{R}^m} (D_ja_{ij}-b_i)^2w^2 dy + \int_{\mathbf{R}^m} cw^2 -fw dy
\end{align*}
By Sobolev inequality, we get
\begin{align*}
    &(\int_{\mathbf{R}^m} |w|^{\frac{2m}{m-2}}dy)^{\frac{m-2}{m}}\le  C_m\int_{\mathbf{R}^m}|Dw|^2dy\\
    \le & \frac{C_m}{\lambda^2}\left(\int_{\mathbf{R}^m}[(D_ja_{ij}-b_i)^2+2\lambda c]w^2 dy -\int_{\mathbf{R}^m}2\lambda fw dy\right)\\
    \le & \frac{C_m}{\lambda^2}\left((\int_{\mathbf{R}^m}|(D_ja_{ij}-b_i)^2+2\lambda c|^{\frac{m}{2}} dy)^{\frac{2}{m}}(\int_{\mathbf{R}^m} w^{\frac{2m}{m-2}})^{\frac{m-2}{m}} + (\int_{\mathbf{R}^m} |f|^{\frac{2m}{m+2}})^{\frac{m+2}{2m}} (\int_{\mathbf{R}^m} |w|^{\frac{2m}{m-2}})^{\frac{m-2}{2m}}\right)
\end{align*}
Assume $w\ne 0$, we get
\begin{equation*}
    \begin{split}
        &(\int_{\mathbf{R}^m} |w|^{\frac{2m}{m-2}}dy)^{\frac{m-2}{2m}}\\
    \le & \frac{C_m}{\lambda^2}\left((\int_{\mathbf{R}^m}|(D_ja_{ij}-b_i)^2+2\lambda c|^{\frac{m}{2}} dy)^{\frac{2}{m}}(\int_{\mathbf{R}^m} w^{\frac{2m}{m-2}})^{\frac{m-2}{2m}} + (\int_{\mathbf{R}^m} |f|^{\frac{2m}{m+2}})^{\frac{m+2}{2m}}\right)
    \end{split}
\end{equation*}
\end{proof}
As a corollary, we have
\begin{corollary}\label{cor: uniqueness of solution}
Let $w\in C^2(\mathbf{R}^m)\cap W^{1,2}(\mathbf{R}^m)\cap L^{\frac{2m}{m-2}}(\mathbf{R}^m)$ be the solution of the equation
\begin{align}\label{eq: linear eq2}
    a_{ij}D_{ij}w+b_iD_i w +cw = 0
\end{align}
Then there exists some $\varepsilon_0$ depending only on $C_m$ and the uniform
elliptic coefficient $\lambda$ such that if  
\[
\int_{\mathbf{R}^m}|(D_ja_{ij}-b_i)^2+2\lambda c|^{\frac{m}{2}} dy\leq \varepsilon_0
\]
Then equation \eqref{eq: linear eq2} only has the trivial solution $w=0$.
\end{corollary}

Finally, we show the following proposition which is from \cite{Liu13}:
\begin{proposition}\label{flatness}
Suppose $(M^n,g)$ and  $\Sigma_t$, $t_0$ as those in Theorem \ref{thm: Existence}. Moreover,  $\Sigma_t$ satisfies the following conditions:
\begin{itemize}
	\item $\Sigma_t$ is isometric to $\mathbb{R}^{n-1}$ and totally geodesic;
	\item $Ric(\nu,\nu)=0$ along $\Sigma_t$, here $\nu$ is the outward unit normal vector of $\Sigma_t$.
	 \end{itemize}	
Then $\{(x,z)\in M: |z|\geq t_0\}$ is flat.
\end{proposition}
\begin{proof}

Let's fix any $\Sigma_t$ in $\{(x,z)\in M: |z|> t_0\}$	 and denote it by $\Sigma$ for simplicity, then we may write 
$$
g= d\rho^2+g_{ij}(\rho, \theta)d\theta^i d\theta^j,
$$
in a tubular neighborhood $\mathcal{U}$ of  $\Sigma$. Here $\rho$ is the distance function to $\Sigma$ in $(M^n,g)$, $\theta=(\theta^1, \cdots, \theta^{n-1})$ is a local coordinates in $\Sigma$, hence, $(\rho, \theta)$ is a local coordinates in $\mathcal{U}$. Then by the first assumption of $\Sigma$ in Proposition \ref{flatness}, we see that for any tangential vectors $X,Y,Z,W$ of $\Sigma$, we have $R(X,Y,Z,W)=0$ along $\Sigma$, and by Coddazi equations, we know that $R(X,Y,Z,\nu)=0$ along $\Sigma$ as well. Hence, it suffices to show $R(\nu,Y,Z,\nu)=0$ along $\Sigma$. To this end,  Let $\Sigma_t\subset \mathcal{U}$ be another leaf in the foliation of these minimal hypersurfaces mentioned in  Theorem \ref{thm: Existence} so that it can be written as graph over $\Sigma$ by a smooth positive function $u$, i.e. 
$$
\Sigma_t:=\{(u(\theta),\theta): \theta\in \Sigma\},
$$
let
$$
f(\rho,\theta)=\rho-u(\theta),
$$
then $f|_{\Sigma_t} =0$. Let $X,Y$ be any tangential vectors of $\Sigma_t$, then
\begin{equation}
	\begin{split}
	\nabla^2f(X,Y)&=\nabla^2_{\Sigma_t}f(X,Y)-<\nabla_XY, \nu>\nu(f)\\
		&=\nabla^2_{\Sigma_t}f(X,Y)\\
		&=0	,
	\end{split}\nonumber
\end{equation}
where $\nabla^2 $,$\nabla^2_{\Sigma_t}$ denote the Hessian operator on $(M,g)$ and $
\Sigma_t$ with the induced metric respectively, and we have utilized the fact that $\Sigma_t$ is totally geodesic,
hence on $\Sigma_t$ we have
\begin{equation}\label{eq:riccati eq}
	\begin{split}
\nabla^2_{\Sigma_t}u(X,Y)&=\nabla^2u(X,Y)\\
&=\nabla^2\rho(X,Y)\\
&=-R(\nu,X,Y,\nu)u+O(|u|^3),
	\end{split}
\end{equation}
where we have utilized the Riccati equation of $\nabla^2\rho(X,Y)$ and the fact that $\Sigma$ is totally geodesic. Finally, let $\theta_0$ be any fixed point in $\Sigma$, then $u(\theta_0)>0$ for each $t$,  and denote $w(\theta)=u^{-1}(\theta_0) u(\theta) $, then by previous estimate, we know that $w$ is uniformly bounded on $\Sigma_t$. Plug these issues into \eqref{eq:riccati eq}, and let $\Sigma_t$ approaches to $\Sigma$, then we get 
$$
\nabla^2_{\Sigma}w=-R(\nu,X,Y,\nu)w,
$$
on $\Sigma$. Taking trace in the above equation and noting that $Ric(\nu,\nu)=0$, we observe that $w$ is a bounded harmonic function on $\Sigma$, hence it must be a positive constant. This implies for any tangential vectors $X,Y$ on $\Sigma$,
$$
R(\nu,X,Y,\nu)=0.
$$
Therefore, the curvature tensor $Rm$ of $(M^n,g)$ vanishes along $\Sigma$, and $\Sigma$ is arbitrary, we see that $\{(x,z)\in M: |z|\geq t_0\}$ is flat.

\end{proof}

\bibliographystyle{alpha}

\bibliography{Positive}

\begin{thebibliography}{LSW63}

\bibitem[AR89]{AR1989}
Michael~T. Anderson and Lucio Rodr\'{\i}guez.
\newblock Minimal surfaces and {$3$}-manifolds of nonnegative {R}icci
  curvature.
\newblock {\em Math. Ann.}, 284(3):461--475, 1989.

\bibitem[Bel23]{Be2023}
Costante Bellettini.
\newblock Extensions of schoen--simon--yau and schoen--simon theorems via
  iteration à la de giorgi, 2023.

\bibitem[Car16]{Carlotto16}
Alessandro Carlotto.
\newblock Rigidity of stable minimal hypersurfaces in asymptotically flat
  spaces.
\newblock {\em Calc. Var. Partial Differential Equations}, 55(3):Art. 54, 20,
  2016.

\bibitem[CCE16]{CCE16}
Alessandro Carlotto, Otis Chodosh, and Michael Eichmair.
\newblock Effective versions of the positive mass theorem.
\newblock {\em Invent. Math.}, 206(3):975--1016, 2016.

\bibitem[CCZ23]{CCZ23}
Shuli Chen, Jianchun Chu, and Jintian Zhu.
\newblock Positive scalar curvature metrics and aspherical summands, 2023.

\bibitem[CM11]{CM11}
Tobias~Holck Colding and William~P. Minicozzi, II.
\newblock {\em A course in minimal surfaces}, volume 121 of {\em Graduate
  Studies in Mathematics}.
\newblock American Mathematical Society, Providence, RI, 2011.

\bibitem[EK23]{EK23}
Michael Eichmair and Thomas Koerber.
\newblock Schoen's conjecture for limits of isoperimetric surfaces, 2023.

\bibitem[Giu84]{Giu1984}
Enrico Giusti.
\newblock {\em Minimal surfaces and functions of bounded variation}, volume~80
  of {\em Monographs in Mathematics}.
\newblock Birkh\"auser Verlag, Basel, 1984.

\bibitem[GLZ20]{Zhou20}
Qiang Guang, Martin Man-chun Li, and Xin Zhou.
\newblock Curvature estimates for stable free boundary minimal hypersurfaces.
\newblock {\em J. Reine Angew. Math.}, 759:245--264, 2020.

\bibitem[GT83]{GT}
David Gilbarg and Neil~S. Trudinger.
\newblock {\em Elliptic partial differential equations of second order}, volume
  224 of {\em Grundlehren der mathematischen Wissenschaften [Fundamental
  Principles of Mathematical Sciences]}.
\newblock Springer-Verlag, Berlin, second edition, 1983.

\bibitem[HY96]{HY1996}
Gerhard Huisken and Shing-Tung Yau.
\newblock Definition of center of mass for isolated physical systems and unique
  foliations by stable spheres with constant mean curvature.
\newblock {\em Invent. Math.}, 124(1-3):281--311, 1996.

\bibitem[Li24]{Li24}
Chao Li.
\newblock The dihedral rigidity conjecture for n-prisms.
\newblock {\em J. Differential Geom.}, 126(1):329--361, 2024.

\bibitem[Liu13]{Liu13}
Gang Liu.
\newblock 3-manifolds with nonnegative ricci curvature.
\newblock {\em Inventiones mathematicae}, 193(2):367--375, 2013.

\bibitem[LSW63]{LSW63}
W.~Littman, G.~Stampacchia, and H.~F. Weinberger.
\newblock Regular points for elliptic equations with discontinuous
  coefficients.
\newblock {\em Ann. Scuola Norm. Sup. Pisa Cl. Sci. (3)}, 17:43--77, 1963.

\bibitem[Sch89]{Schoen1989}
Richard~M. Schoen.
\newblock Variational theory for the total scalar curvature functional for
  {R}iemannian metrics and related topics.
\newblock 1365:120--154, 1989.

\bibitem[Sim68]{Simon1968}
James Simons.
\newblock Minimal varieties in riemannian manifolds.
\newblock {\em Ann. of Math.}, 88(2):62--105, 1968.

\bibitem[Sim83]{Simon83}
Leon Simon.
\newblock {\em Lectures on geometric measure theory}, volume~3 of {\em
  Proceedings of the Centre for Mathematical Analysis, Australian National
  University}.
\newblock Australian National University, Centre for Mathematical Analysis,
  Canberra, 1983.

\bibitem[SS81]{Schoen1981}
Richard Schoen and Leon Simon.
\newblock Regularity of stable minimal hypersurfaces.
\newblock {\em Comm. Pure Appl. Math.}, 34(6):741--797, 1981.

\bibitem[SY79]{SY79}
Richard Schoen and Shing~Tung Yau.
\newblock On the proof of the positive mass conjecture in general relativity.
\newblock {\em Comm. Math. Phys.}, 65(1):45--76, 1979.

\end{thebibliography}
\end{document}